\documentclass[reqno]{amsart}
\usepackage{amsfonts}
\usepackage{mathrsfs}
\usepackage{amsmath}
\usepackage{amsthm}
\usepackage{amssymb}
\usepackage{tikz}
\usepackage{array}
\usepackage{pifont}
\usepackage{verbatim}
\usepackage{enumitem}

\newcommand{\mr}[1]{\mathscr{#1}}
\def\H{\mathcal{H}}
\def\V{\mathcal{V}}

\def\R{\mathbb{R}}
\def\N{\mathbb{N}}

\def\and{\qquad\text{and}\qquad}

\def\({\left(}
\def\){\right)}
\def\eps{\varepsilon}
\def\p{\partial}
\def\al{\alpha}
\def\be{\beta}

\def\wto{\rightharpoonup}

\def\la{\langle}
\def\ra{\rangle}
\def\td{{\rm d}}

\def\be{\begin{equation}}
\def\ee{\end{equation}}

\def\ben{\begin{eqnarray}}
\def\een{\end{eqnarray}}

\makeatletter
\@ifundefined{c@chapter}{
\newcounter{cnum}
}{
\newcounter{cnum}[chapter]
}
\def\C{\@ifnextchar[{\@with}{\@without}}
\def\@with[#1]{\@ifundefined{c@#1}{%
            \ifnum\thecnum<1
            \stepcounter{cnum}
            \fi
           \newcounter{#1}
           \setcounter{#1}{\thecnum}
            C_{\thecnum}
           \stepcounter{cnum}
        }{
        \ifnum\the\csname c@#1\endcsname<\thecnum
         C_{{\the\csname c@#1\endcsname}}
        \else
         C_{{\the\csname c@#1\endcsname}}
         \ifnum\thecnum<1
         \stepcounter{cnum}
         \fi
        \stepcounter{cnum}
        \fi
        }}
\def\@without{\ifnum\thecnum<1
                \stepcounter{cnum}
              \fi
              C_{\thecnum}
              \stepcounter{cnum}
        }
\makeatother

\newtheorem{thm}{Theorem}[section]
\newtheorem{lemma}[thm]{Lemma}

\theoremstyle{definition}
\newtheorem{Def}[thm]{Definition}

\theoremstyle{remark}
\newtheorem{remark}[thm]{Remark}

\numberwithin{equation}{section}

\begin{document}

\title[wave equation with nonlinear damping and super-cubic nonlinearity]{Well-posedness and global attractor for wave equation with nonlinear damping and super-cubic nonlinearity}

\author{Cuncai Liu}
\address{School of Mathematics and Physics, Jiangsu University of Technology, Changzhou, 213001,China}
\email{liucc@jsut.edu.cn}
\thanks{This work was supported by the NSFC(11701230,11801227,11801228,12026431,12442050), Natural Science Fund For Colleges and
  Universities in Jiangsu Province (22KJD110001), Jiangsu 333 Project, QingLan Project of Jiangsu Province and Zhongwu Young Innovative Talents Support Program.}

\author{Fengjuan Meng}
\address{School of Mathematics and Physics, Jiangsu University of Technology, Changzhou, 213001,China}
\email{fjmeng@jsut.edu.cn}
\author{Xiaoying Han}
\address{Department of Mathematics and Statistics, 221 Parker Hall, Auburn University,
  Auburn, AL 36849, USA}
\email{xzh003@auburn.edu.cn}
\author{Chang Zhang}
\address{School of Mathematics and Physics, Jiangsu University of Technology, Changzhou, 213001,China}

\email{chzhnju@126.com}




\begin{abstract}
  This study investigates a semilinear wave equation characterized by nonlinear damping $g(u_t) $ and nonlinearity $f(u)$. First, the well-posedness of weak solutions across broader exponent ranges for $g$ and $f$ is established, by utilizing a priori space-time estimates. Moreover, the existence of a global attractor in the phase space $H^1_0(\Omega)\times L^2(\Omega)$ is obtained. Furthermore, it is proved that this global attractor is regular, implying that it is a bounded subset of $(H^2(\Omega)\cap H^1_0(\Omega))\times H^1_0(\Omega)$.
\end{abstract}

\maketitle
\setcounter{tocdepth}{1}
\tableofcontents

\numberwithin{equation}{section}

\section{Introduction}
This paper addresses the nonlinear damped wave equation governed by:
\begin{align}\label{eq1}
  \begin{cases}u_{tt}+g(u_{t})-\Delta u+f(u)=\phi, \quad \quad \qquad  (t,x)\in \mathbb{R}_{+}\times\Omega, \\
    u(0, x)=u_0(x),~~ u_{t}(0, x)=u_1(x),  \quad \,\, \,\, \,\, x\in \Omega,                     \\
    u(t, x)=0, \qquad \qquad \qquad \quad \quad \qquad \,\,\,  (t,x)\in \mathbb{R}_{+}\times \partial\Omega.
  \end{cases}
\end{align}
where $\Omega$ denotes a bounded domain in $\mathbb{R}^3$ with smooth boundary, and the external force term $\phi\in L^2(\Omega)$ remains independent of time.

To study the dynamics of this system, we impose certain conditions on the nonlinear damping  $g$ and the nonlinearity $f$. Specifically, it is assumed that
\begin{itemize}
  \item[({\bf G})] The damping function $g\in C^1(\mathbb{R})$ is strictly increasing and satisfies $g(0)=0$ and the growth condition
    \begin{equation}\label{G1}
      c_1(|s|^{m}-c_2)\le |g(s)| \le c_3(|s|^{m}+1),~~ 1<m\le5.\tag{G1}
    \end{equation}
  \item[({\bf F})] The nonlinearity $f\in C^2(\mathbb{R})$ and satisfies $f(0)=f^{\prime}(0)=0$ and the growth condition
    \begin{equation}\label{F1}
      |f^{\prime\prime}(s)|\leq c(1+|s|^{p-2}),~~ 2\le p\le \min\{5,3m\},\tag{F1}
    \end{equation}
    alongside the dissipative condition
    \begin{equation}\label{F2}
      \varliminf\limits_{|s|\rightarrow\infty}\frac{f(s)}{s}> -\lambda,\tag{F2}
    \end{equation}
    with $\lambda>0$ being less than the first eigenvalue $\lambda_1$ of $-\Delta$ with Dirichlet boundary condition on $\Omega$.
    Furthermore, we introduce a lower bound estimate on $f^\prime$, ensuring that  for any $\omega>0$ there exists a constant $K_\omega>0$ such that
    \begin{equation}\label{F3}
      f^{\prime}(s)\ge -\omega s^4 -K_\omega, ~~\forall s\in \R.\tag{F3}
    \end{equation}
\end{itemize}

It is clear that condition \eqref{F3} becomes trivial when  $p<5$, due to the stipulations of \eqref{F1}, rendering \eqref{F3} specifically restrictive only at $p=5$. Defining $F(s)$ as $F(s) = \int_0^s f(\tau) \text{d}\tau$, within the constraints of the aforementioned assumptions  we derive the ensuing inequalities:
\begin{equation}\label{fl-Flu}
  f(s)s\ge -\lambda s^2-C, ~~ -\dfrac{\lambda}{2}s^2-C\le F(s) \le C(1+|s|^{p+1}),
\end{equation}
and further
\begin{equation}\label{fs-F}
  f(s)s-F(s)=\int_0^s f^\prime(s)s\td s\ge -\omega s^6 -K_\omega s^2.
\end{equation}
These formulations underscore the relationship between the function $f$ and its integral representation $F(s)$, delineating the bounds and behavior influenced by the parameters $\lambda$, $\omega$, and $K_\omega$.

The well-posedness of problem \eqref{eq1} has been extensively discussed, with a rich body of literature addressing various parameter assumptions. A summary of some pivotal results is as follows. Lions and Strauss \cite{Lions65} demonstrated the existence of a weak solution in $\mathbb{R}^3$ for cases where $2 \leq p \leq m$, also introducing a monotonicity method to address the convergence of the nonlinear damping term. Feireisl \cite{feireisl95_nldamp} explored the well-posedness of problem \eqref{eq1} under a slightly different assumption:
\begin{equation}
  \label{G1'}
  c_1(|s|^r - c_2) \leq |g(s)| \leq c_3(1 + |s|^m)\tag{G1$^{\prime}$},
\end{equation}
with $p < \frac{6r}{r+1}$, and $1 \leq r \leq m < 5$. This assumption implies that the upper growth exponent $m$ for $g$ may not match the lower growth exponent $r$, leading to the absence of an appropriate space for the weak solution definition to guarantee the integrability of the damping term $g(u_t)$. To address this issue, Feireisl established that a weak solution can be considered as the limit of a sequence of strong solutions, thus proving the well-posedness of both weak and strong solutions within a bounded domain. Chueshov and Lasiecka \cite{chueshov04} examined the weak solution for $m<5$, $p<3$, and $r=1$, particularly when the damping parameter is large with $p=3$. In scenarios involving higher growth nonlinearity, more sophisticated estimates are necessary. For instance, in the linear damping case (i.e., $g(u_t) = \gamma u_t$), Feireisl \cite{feireisl95} confirmed well-posedness through the $L^4(0,T;L^{12}(\mathbb{R}^3))$ Strichartz estimates. However, applying the Strichartz estimate to secure the well-posedness of \eqref{eq1} is challenging due to the introduction of a nonlinear damping term. Recently, Todorova \cite{todorova15} achieved a similar space-time estimate by selecting an appropriate nonlinear test function for $m = \frac{5}{3}$. This led to the establishment of the well-posedness of the weak solution in $\mathbb{R}^3$, along with the proof of strong solution well-posedness on $\mathbb{R}^3$ for $m \geq 2$.

Regarding the long-term behavior of problem \eqref{eq1}, the foundational works in this area are notably attributed to \cite{cl84}, \cite{hale88}, and \cite{haraux87}. Raugel \cite{rau92} was instrumental in identifying the global attractor under the constraint $m_1 \leq g'(s) \leq c(1 + |s|^{2/3})$, with $m_1$ being sufficiently large. Feireisl \cite{feireisl95_nldamp} verified the existence of a global attractor for the weak solution, predicated on the condition \eqref{G1'}, where $p < 6r/(r+1)$, $1 \leq r \leq m < 5$. The work in \cite{94fei} explored the finite dimensionality of attractors for wave equations with nonlinear damping in one-dimensional scenarios. Subsequent studies by \cite{chueshov02} and \cite{chueshov04} delved into the existence and characteristics (such as structure and dimension) of global attractors in situations where $m < 5$, $p < 3$, and $r = 1$, especially with a significantly large damping parameter at $p = 3$. These studies offered a comprehensive abstract framework for analyzing the asymptotic behavior of solutions to evolutionary equations with a second-order time term, as detailed further in \cite{chueshov08}. Sun \cite{sun06} later demonstrated the existence of a global attractor for the weak solution in cases with $r = 1$, $p \leq 3$, and $m < 5$. When the conditions were set to $p < 3$, $r = m \leq 5$, Nakao \cite{nakao06} established the polynomial absorbing rate of the absorbing ball and the existence of a global attractor. With the help of a compensated compactness criterion developed in \cite{Khanmamedov06_karman} for asymptotic smoothness, Khanmamedov \cite{khanmamedov06} established the existence of the global attractor for the weak solution where $r = 1$, $p \leq 3$, and $m \leq 5$.

Regarding the regularity of the global attractor, specifically in the 2D scenario, Lasiecka \cite{la02} and Khanmamedov \cite{khanmamedov09} demonstrated the $(H^2(\Omega)\cap H^1_0(\Omega))\times H^1_0(\Omega)$-regularity of the global attractor for the weak and strong solutions of \eqref{eq1}, respectively. Chueshov and Lasiecka \cite{chueshov04} established that the global attractor is bounded in $(H^2(\Omega)\cap H^1_0(\Omega))\times H^1_0(\Omega)$ when $p < 3$ and $m < 5$. In \cite{khanmamedov10}, Khanmamedov introduced the ``finite cutting-off and decomposition method'' to examine the regularity for cases where $p \leq 3$ and $m \leq 5$. This method has proven effective in investigating the regularity of the global attractor for wave equations with different types of nonlinear damping, such as those involving displacement-dependent damping \cite{Khanmamedov10-dis}.

From the discussion above, it is evident that the well-posedness and long-term behavior of problem \eqref{eq1} are significantly influenced by the growth exponents $m$ of the damping term $g$ and  $p$ of the nonlinearity $f$. For clarity, we present these findings in Figure \ref{fig:1}, illustrating that the weak solution of problem \eqref{eq1} is well-posed in regions I and II, and at a single point ($p=5$, $m=\frac{5}{3}$). The existence of the global attractor has been investigated in regions I and II, while the regularity of the global attractor has been confirmed only in region I.

\begin{figure}[htpb!]
  \centering
  \begin{tikzpicture}
    [cube/.style={very thick,black},
      grid/.style={very thin,gray},
      axis/.style={->,black,thick}]

    \draw[axis] (0,0) -- (6,0) node[anchor=west]{$p$};
    \draw[axis] (0,0) -- (0,5.5) node[anchor=west]{$m$};
    \coordinate (A) at (1,1);

    \draw (1,0) -- (1, 0.1);
    \node[below] at (1,0){1};

    \draw (3,0) -- (3, 0.1);
    \node[below] at (3,0){3};

    \draw (5,0) -- (5, 0.1);
    \node[below] at (5,0){5};

    \draw (0,1) -- (0.1, 1);
    \node[left] at (0,1){1};

    \draw (0,5/3) -- (0.1,5/3);
    \node[left] at (0,5/3){5/3};

    \draw (0,5) -- (0.1, 5);
    \node[left] at (0,5){5};

    \draw[thick, fill=black!10] (1,5) -- (1,1) -- (3,1) -- (3,5);
    \fill[black!20, domain=3:5] (3,5) -- (3,1) -- plot(\x, {\x/(6-\x)});

    \fill[black!30, domain=3:5] plot(\x, {\x/(6-\x)}) -- (5,5/3)--(3,1);

    \draw[thick] (1,1) -- (1,5);
    \draw[thick] (1,1) -- (5,1);

    \draw[thick, dashed] (5,1) -- (5,5/3);
    \draw[thick] (1,5) -- (5,5);

    \draw[thick] (3,1) -- (5,5/3) -- (5,5);
    \draw[thick] (3,1) -- (3,5);
    \draw[thick, domain=3:5]
    plot(\x, {\x/(6-\x)});
    \fill (5,5/3) circle (2pt);

    \node[font=\fontsize{8}{8}] at (2,2.5){I};
    \node[font=\fontsize{8}{8}] at (3.7,2.5){II};
    \node[font=\fontsize{8}{8}] at (4.7,2.5){III};
    \node[font=\fontsize{8}{8}] at (3.9,4){$p < \frac{6m}{m+1}$};
    \node[font=\fontsize{8}{8}] at (4.5,2) {$p\le 3m$};
  \end{tikzpicture}
  \caption{Well-posedness region of the weak solution}
  \label{fig:1}
\end{figure}
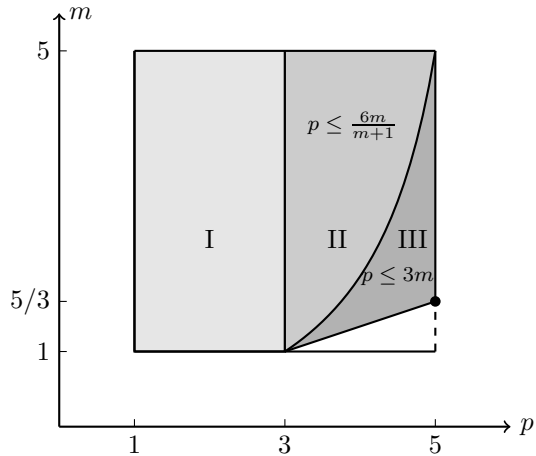

The objective of this paper is to investigate the well-posedness and long-term behavior of problem \eqref{eq1}, expanding the scope to include broader ranges of exponents for the nonlinear damping $g$ and the nonlinearity $f$. Specifically, we aim to extend the existing results concerning the well-posedness and the global attractor for the weak solution of problem \eqref{eq1} to the shaded region III presented in Figure 1. This pertains to the range where $p\le \min\{3m,5\}$ and $m\le 5$. Our approach begins with establishing a priori space-time estimates for the weak solution, employing Todorova's methodology. These estimates allow us to address the challenge posed by the inability to apply Strichartz estimates directly to wave equations with nonlinear damping, leading to the establishment of the weak solutions' continuous dependence on initial conditions. Subsequently, we demonstrate the asymptotic compactness of the semigroup generated by the weak solution using the energy method in \cite{ball04,mrw98}, which ensures the existence of a global attractor. Finally, by refining the finite cut-off and decomposition method, we identify a regular point within the global attractor that belongs to the space $(H^2(\Omega)\cap H^1_0(\Omega))\times H^1_0(\Omega)$. This discovery, coupled with the uniform boundedness of the strong solutions, implies that the global attractor is, in fact, bounded in the space $(H^2(\Omega)\cap H^1_0(\Omega))\times H^1_0(\Omega)$.

It is worth mentioning that the a priori space-time estimate of weak solutions is primarily influenced by the subdued growth of the nonlinear damping term $g$. Yet, as the growth exponent $m$ approaches 1, this estimate becomes insufficient for addressing the increased growth of the nonlinearity $f$. Consequently, a gap emerges, represented by the uncharted ``triangle region'' in Figure \ref{fig:1}. While complete resolution within this region remains elusive, we propose an alternative methodology to achieve partial insights in this work, the details of which will be explored in a forthcoming paper.

The rest of the paper is organized as follows. Section 2 introduces essential preliminaries, including notations and foundational inequalities, and establishes more inequalities to be used in the sequel. Section 3 is dedicated to developing the a priori space-time estimates and establishing the well-posedness of both weak and strong solutions. The existence of the global attractor is the focus of Section 4. In the end the regularity of the global attractor in $(H^2(\Omega)\cap H^1_0(\Omega))\times H^1_0(\Omega)$ is shown in Section 5.

\section{Preliminaries}

Regarding the phase spaces pertinent to our study, we define spaces $$\H=H^1_0(\Omega)\times L^2(\Omega), \quad \V=(H^2(\Omega)\cap H^1_0(\Omega))\times H^1_0(\Omega),$$
equipped with norms
$$\|(u,v)\|_{\H}=\left(\|\nabla u\|^2+\|v\|^2\right)^{\frac12}\text{~and~}\|(u,v)\|_{\V}=\left(\|\Delta u\|^2+\|\nabla v\|^2\right)^{\frac12}$$
respectively.

For ease of reference and simplicity, the inner product and norm in $L^2(\Omega)$ are represented by $\langle\cdot,\cdot\rangle$ and $\|\cdot\|$, respectively. Meanwhile, the norm in $L^q(\Omega)$ is denoted as $\|\cdot\|_{q}$.

Note that under the assumption ({\bf G}), the integrability of damping term $g(u_t)$ can be obtained by the $L^{m+1}([0,T]\times\Omega)$ integrability of $u_t$. Hence, the weak solution can be defined as follows.
\begin{Def}[Weak solution]
  A function $u$ satisfies
  $$u\in C([0,T];H^1_0(\Omega)),~~ u_t\in C([0,T];L^2(\Omega))\cap L^{m+1}([0,T]\times\Omega)$$
  possessing the properties $u(0)=u_0$ and $u_t(0)=u_1$ is said to be a
  weak solution of problem \eqref{eq1} on $[0,T]\times\Omega$, if and only if equation \eqref{eq1} is satisfied in the sense of distribution, \textit{i.e.},
  $$\int_{0}^{T}\int_\Omega \left(-u_t\psi_t+g(u_t)\psi+\nabla u\nabla \psi+f(u)\psi\right)\td x \td t=\int_{0}^{T}\int_\Omega \phi\psi\td x\td t$$
  for any $\psi \in C_0^{\infty}((0,T)\times \Omega)$.
\end{Def}

It is known that  $L^2(\Omega)$ possesses a complete orthonormal basis, $\{e_i\}_{i\in \mathbb{N}}$, comprised of eigenvectors of $-\Delta$ under Dirichlet boundary conditions:
$$-\Delta e_i=\lambda_ie_i,~~ i\in \mathbb{N},$$
where the sequence $\lambda_i$ is monotonically increasing and diverges to $+\infty$. Given the smoothness of $\Omega$, each $e_i$ is sufficiently smooth. For any real number $s$, we introduce a hierarchy of Hilbert spaces, which are compactly nested:
$$  H^s=D((-\Delta)^\frac{s}{2})=\left\{w=\sum_{i=1}^{\infty}a_ie_i\mid \sum_{i=1}^{\infty}\lambda_i^sa_i^2<+\infty\right\},$$
equipped with the norm $\|\cdot\|_{H^s}$ defined by
$$\|w\|_{H^s}=\|(-\Delta)^{\frac{s}{2}}w\|=\left(\sum_{i=1}^{\infty}\lambda_i^sa_i^2\right)^{\frac12}, \quad w \in H^s.$$

Throughout this paper, the symbol $C$ represents a generic constant, which is occasionally indexed for clarity. Additionally, various positive constants denoted by $C_i, i\in \mathbb{N}$, are employed to facilitate specific distinctions within analysis.

Next, we proceed to establish some key inequalities pivotal for the subsequent analysis.

\begin{lemma} Under the Assumption {\rm({\bf G})}, for any $\delta\in (0,1)$, there exists constant $c(\delta)>0$ such that
  \begin{equation} \label{g_delta}
    s^2+|s|^{m+1}\le c(\delta)g(s)s+\delta \text{~~ and ~~} |g(s)|^{(m+1)/m}\le c(\delta)g(s)s+\delta.
  \end{equation}
\end{lemma}
\begin{proof}
  First, since $g\in C^1(\mathbb{R})$ is monotonically increasing and $g(0)=0$, then $g(s)s\ge 0$. It follows directly from the inequality \eqref{G1} that
  $$\liminf_{s\to \infty}\frac{g(s)s}{|s|^{m+1}}\ge c_1.$$
  Set
  $$C_1(\delta): =\inf_{|s|\ge \delta/2}\left(\frac{g(s)s}{|s|^{m+1}}+\frac{g(s)s}{s^2}\right).$$
  Then $C_1(\delta) > 0$ and satisfies
  \begin{equation}\label{new:1}
    s^2+|s|^{m+1}\le C_1(\delta)g(s)s,~~ \text{for all  }|s|\ge \delta/2.
  \end{equation}
  On the other hand, when $|s|\le \delta/2<1/2$,
  \begin{equation}\label{new:2} s^2+|s|^{m+1}\le 2s^2\le 2|s|\le \delta.
  \end{equation}
  The first inequality in \eqref{g_delta} is an immediate consequence of \eqref{new:1} and \eqref{new:2}.

  Next, note that because $g(s)$ is increasing,  $g(s)s$ is decreasing on $(-\infty, 0)$ and increasing on $[0, +\infty)$. Therefore there exists a unique $s^*>0$ such that $$\min \{g(s^*),|g(-s^*)|\}=\delta^{m/(m+1)},$$  which implies that
  \begin{equation}\label{new:3}
    |g(s)|^{(m+1)/m}\le \delta, ~~ \text{for all  } |s|\le s^*.
  \end{equation}
  On the other hand, setting
  $$C_2(\delta)=\inf_{|s|\ge s^*}\frac{g(s)s}{|g(s)|^{(m+1)/m}},$$
  then $C_2(\delta) > 0$ and
  \begin{equation}\label{new:4}
    |g(s)|^{(m+1)/m }\le C_2(\delta)g(s)s,~~ \text{for all  }|s|\ge s^*.
  \end{equation}

  The second inequality in \eqref{g_delta} follows immediately from \eqref{new:3} and \eqref{new:4}, and the proof is complete.
\end{proof}

The Lemma \ref{inter_ineq} below gives an interpolation inequality.
\begin{lemma} \label{inter_ineq}
  Suppose that parameters $\lambda$, $\mu$, $k$ and $q$ satisfy
  \begin{itemize}
    \item[(i)]  $\lambda, \mu>0$, $k\ge 1$ and $1\le q\le 3$;
    \item[(ii)] $\dfrac{\lambda}{3k+3}+\dfrac{\mu}{m+1}\le 1$;
    \item[(iii)] $\dfrac{\lambda-(q-1)(k-1)}{k+5}+\dfrac{\mu}{m+1}\le 1$.
  \end{itemize}
  Then given any small constant $\theta>0$, there exists a constant $C = C(\theta)>0$ such that the  inequality
  \begin{equation}\label{interp}
    \left(\int_\Omega |\varphi|^{\lambda} |\psi|^{\mu}\td x\right)^{1/q}\le \theta+\theta \|\varphi\|^{k+1}_{3k+3}+C \left(1+\|\varphi\|^{\lambda(m+1)/\mu}_{6}\right)\|\psi\|^{m+1}_{m+1}
  \end{equation}
  holds for all $\varphi\in L^{3k+3}(\Omega)$ and $\psi\in L^{m+1}(\Omega)$.
\end{lemma}
\begin{proof}
  First, by virtue of H\"older's inequality, we have
  \begin{equation}\label{uut}
    \int_\Omega |\varphi|^\lambda |\psi|^\mu\td x \le \left(\int_\Omega |\varphi|^{\frac{\lambda(m+1)}{m+1-\mu}}\td x\right)^{\frac{m+1-\mu}{m+1}}\|\psi\|^\mu_{m+1}.
  \end{equation}
  Noticing that $\dfrac{\lambda(m+1)}{m+1-\mu} \le 3k+3$ due to the condition (ii), and applying H\"older's inequality to the first term on the right hand side of inequality \eqref{uut} gives
  \begin{equation}\label{ieq_alpha}
    \int_\Omega |\varphi|^{\frac{\lambda(m+1)}{m+1-\mu}}\td x\le C_1 \|\varphi\|^{\frac{\alpha\lambda(m+1)}{m+1-\mu}}_{3k+3}\|\varphi\|^{\frac{(1-\alpha)\lambda(m+1)}{m+1-\mu}}_{6},
  \end{equation}
  where $\alpha=\max\left\{0, \left(\dfrac16-\dfrac{m+1-\mu}{\lambda(m+1)}\right)/\left(\dfrac16-\dfrac{1}{3k+3}\right)\right\}$.

  Next, substituting \eqref{ieq_alpha} into \eqref{uut}, then applying Young's inequality, one gets
  \begin{equation}\label{2.4}
    \begin{aligned}
      \left(\int_\Omega |\varphi|^\lambda |\psi|^\mu\td x\right)^{1/q} & \le C_2 \|\varphi\|^{\alpha\lambda/q}_{3k+3}\|\varphi\|^{(1-\alpha)\lambda/q}_{6}\|\psi\|^{\mu/q}_{m+1}                                              \\
                                                                       & \le \theta \|\varphi\|^{\frac{\alpha \lambda (m+1)}{q(m+1)-\mu}}_{3k+3}+C(\theta) \|\varphi\|^{(1-\alpha)\lambda (m+1)/\mu}_{6}\|\psi\|^{m+1}_{m+1}.
    \end{aligned}
  \end{equation}
  Further,  since condition (iii) is equivalent to $(\lambda-6)(m+1)+6\mu \le (k-1)(q(m+1)-\mu)$, then
  $$\frac{\alpha\lambda (m+1)}{q(m+1)-\mu}=\max\left\{0,\frac{(\lambda-6)(m+1)+6\mu}{q(m+1)-\mu}\frac{k+1}{k-1}\right\}\le k+1.$$
  Applying Young's inequality again to the terms on the right side of \eqref{2.4} gives
  $$ \theta \|\varphi\|^{\frac{\alpha\lambda (m+1)}{q(m+1)-\mu}}_{3k+3} \le \theta +  \theta \|\varphi\|^{k+1}_{3k+3}\and \|\varphi\|^{(1-\alpha)\lambda (m+1)/\mu}_{6}\le 1+\|\varphi\|^{\lambda (m+1)/\mu}_{6},$$
  that implies the desired assertion \eqref{interp} alongside \eqref{2.4}. The proof is complete.
\end{proof}

We also need the following Gronwall-type lemma, the proof of which can be found in \cite[Lemma 2.2]{grasselli04}, for instance.

\begin{lemma}[Gronwall's lemma]\label{gronwall}
  Let $\Phi(t)$ be an absolutely continuous function on $[0,\infty)$. Assume that there exists some $\al>0$ such that the differential inequality
  $$\frac{\td}{\td t}\Phi(t) +2\al \Phi(t)\le h_1(t)\Phi(t)+h_2(t)$$
  holds for almost every $t\in [0,\infty)$, where $h_1(t)$ and $h_2(t)$ are functions on $[0,\infty)$ satisfying
  $$\int_\tau^t |h_1(s)|\td s\le \al (t-\tau)+a, ~~ \sup_{t>0}\int_t^{t+1}|h_2(s)|\td s\le b$$
  for some $a, b\ge 0$. Then there exists $\mu=\mu(h_1)\ge 1$ such that
  $$\Phi(t)\le \mu |\Phi(0)|e^{-\al t}+b\mu e^{2\al}/(e^{\al}-1), \qquad \forall \,\, t\in [0,\infty).$$
\end{lemma}

\section{Well-posedness of solutions}
This section is dedicated to the well-posedness of solutions for the problem \eqref{eq1}. The process begins with an a priori space-time estimate of the weak solution, which is essential for proving the existence and uniqueness of such solutions. This estimate, as expressed in Proposition \ref{priori} below, enables the establishment of the weak solution's uniqueness. To facilitate the construction of the weak solution and to prove the regularity of the global attractor, we then focus on establishing the existence of strong solutions, beginning with a recall of their definition. Finally, the existence of weak solutions is derived by the approximation of these strong solutions.

\subsection{The a priori estimate of weak solution}

In this subsection, we establish a priori space-time estimates for the weak solutions to \eqref{eq1}, which are crucial for proving the existence and uniqueness of these solutions.

\begin{lemma}[The a priori estimate of weak solutions]\label{priori}
  Let Assumptions {\rm({\bf G})} and {\rm({\bf F})} hold, and $\phi\in L^2(\Omega)$. Then any weak solution of the problem \eqref{eq1} satisfies the following space-time estimate
  \begin{equation}\label{space-time}
    \int_0^T \|u\|^{k+1}_{3k+3}\td t\le C(R^7+R^{6}T),
  \end{equation}
  where $k=\min\{5/m,3m-2\}$, $R=\sup_{0\le s \le T}(\|u(s)\|_6+\|u_t(s)\|)+\int_0^T\|u_t\|^{m+1}_{m+1}\td t+\|\phi\|+1$, and the constant $C$ is independent of $u$ and $T$.
\end{lemma}
\begin{proof}
  The conclusion is trivial in the case of $k=1$, so we will assume $k>1$  for the remainder of the proof. To obtain the space-time estimate, inspired by Todorova \cite{todorova15}, we will construct a nonlinear test function.

  Given any $\eps\in (0,1)$, let $M_\eps(s):=\dfrac{|s|^{k-1}}{1+\eps|s|^{k-1}}s$. Then
  $$M_\eps^\prime(s)=\dfrac{k+\eps|s|^{k-1}}{(1+\eps|s|^{k-1})^2}|s|^{k-1},$$
  and the following estimates hold:
  \begin{equation}\label{H_uni}
    |M_\eps(s)|\le \min\left\{|s|^k, \frac{|s|}{\eps}\right\}, \quad  0\le M_\eps^\prime(s)\le k \min \left\{|s|^{k-1},\frac{1}{\eps}\right\}.
  \end{equation}
  This implies that $M_\eps(u)\in L^\infty(0,T;H^1_0(\Omega))$, which allows us to multiply the equation \eqref{eq1} by $M_\eps(u)$, and then integrate over $[0,T]\times \Omega$. The resulting equality reads
  \begin{equation}\label{eqh}
    \int_0^{T}\int_\Omega\left(\la M_\eps(u),u_{tt}\ra_{H^{1}_0(\Omega)}+M_\eps^\prime(u)|\nabla u|^2+M_\eps(u)(g(u_t)+f(u)-\phi)\right)\td x\td t=0,
  \end{equation}
  where $\la \cdot,\cdot \ra_{H^{1}_0(\Omega)}$ is the duality product in $H^{1}_0(\Omega)$. Approximating the function $u$ by smooth ones and arguing in a standard way (see \cite{temam97}), we conclude that
  \begin{equation}\label{eqh:new}
    \begin{aligned}
        & \int_0^{T}\int_\Omega M_\eps^\prime(u)|\nabla u|^2 \td x \td t                                                \\
      = & \int_0^{T}\int_\Omega \big[ M_\eps^\prime(u)|u_t|^2 \td x\td t - M_\eps(u)(g(u_t)+f(u)-\phi) \big] \td x\td t \\
        & +  \int_\Omega \left( M_\eps(u(0))u_{t}(0) - M_\eps(u(T))u_{t}(T) \right) \td x.
    \end{aligned}
  \end{equation}
  To derive the space-time estimate, we estimate each term on both sides of \eqref{eqh:new}. For simplicity of notations, throughout the proof set  $$U_\eps: =\dfrac{|u|}{({1+\eps|u|^{k-1}})^{\frac1{k+1}}}.$$

  First, notice that since $k > 1$,
  $$M_\eps^\prime(u) \geq \frac{|u|^{k-1}}{1+\eps|u|^{k-1}},$$
  and thus
  \begin{equation*}
    M_\eps^\prime(u)|\nabla u|^2\ge \frac{|u|^{k-1}}{1+\eps|u|^{k-1}}|\nabla u|^2=|\nabla m_\eps(|u|)|^2,
  \end{equation*}
  where
  \begin{equation*}
    \begin{aligned}
      m_\eps(|u|) & =\int_0^{|u|}\frac{s^{\frac{k-1}2}}{({1+\eps s^{k-1}})^{\frac12}}\td s  \ge \frac1{({1+\eps|u|^{k-1}})^{\frac12}}\int_0^{|u|}s^{\frac{k-1}2}\td s \\&=\frac{2|u|^{\frac{k+1}2}}{(k+1)({1+\eps|u|^{k-1}})^{\frac12}}
      =\frac{2}{k+1} U_\eps^{\frac{k+1}2}.
    \end{aligned}
  \end{equation*}
  It  then follows by the Sobolev embedding inequality that
  \begin{equation}\label{hd}
    \int_\Omega M_\eps^\prime(u)|\nabla u|^2\td x\ge \|\nabla m_\eps(|u|)\|^2\ge C \|m_\eps(|u|)\|^{2}_{6}\ge C \|U_\eps\|^{k+1}_{3k+3}.
  \end{equation}

  Next, recall that $k=\min\{5/m,3m-2\}$ implies $k \le 3$. Then by \eqref{H_uni} we have
  \begin{equation}\label{hut}
    \left\vert\int_\Omega M_\eps(u)u_{t}\td x\right\vert \le \|M_\eps(u)\|\cdot\|u_t\|\le \|\,|u|^k\|\cdot\|u_t\|\le C \|u\|_{6}^{k}\|u_{t}\|\le C R^4,
  \end{equation}
  and similarly,
  \begin{equation}\label{hphi}
    \left|\int_\Omega M_\eps(u)\phi \td x\right| \le \|M_\eps(u)\|\cdot\|\phi\|\le \|\,|u|^k\|\cdot\|\phi\|\le C \|u\|_{6}^{k}\|\phi\|\le C R^4.
  \end{equation}

  The first term  and  the term with nonlinear damping on the right hand side of \eqref{eqh:new} can be estimated in a similar manner. More precisely,  notice that
  \begin{eqnarray*}
    M_\eps^\prime(u) &=& |U_\eps|^{k-1} \frac{k + \eps |u|^{k-1}}{({1+\eps|u|^{k-1}})^{\frac{k+3}{k+1}}} \le \frac{k}{({1+\eps|u|^{k-1}})^{\frac2{k+1}}} |U_\eps|^{k-1}\le k |U_\eps|^{k-1}, \\
    |M_\eps(u)| &=& \frac{1}{({1+\eps|u|^{k-1}})^{\frac1{k+1}}}|U_\eps|^k\le |U_\eps|^k .
  \end{eqnarray*}
  It follows directly that
  \begin{equation}\label{H'ut2}
    \int_\Omega M_\eps^\prime(u)|u_{t}|^2\td x\le k\int_\Omega |U_\eps|^{k-1}|u_{t}|^2\td x,
  \end{equation}
  and also due to condition \eqref{G1} that
  \begin{equation}\label{Hg}
    \left\vert\int_\Omega M_\eps(u)g(u_t) \td x\right\vert\le C \int_\Omega |U_\eps|^k (1+|u_t|^m) \td x.
  \end{equation}
  Taking into account the fact
  \begin{equation*}
    \frac{k}{k+5}+\frac{m}{m+1}\le 1 \text{~~  and ~~} \frac{k-1}{k+5}+\frac{2}{m+1} \le 1,
  \end{equation*}
  and applying Lemma \ref{inter_ineq} with exponents $\lambda=k-1>0$, $\mu=2$, and $\lambda=k$, $\mu=m$  to \eqref{H'ut2} and \eqref{Hg} respectively, we obtain that the estimates
  \begin{eqnarray}\label{hp}
    \int_\Omega M_\eps^\prime(u)|u_t|^2\td x & \le& \theta+\theta\|U_\eps\|^{k+1}_{3k+3}+C_\theta (1+\|U_\eps\|^{\frac{(k-1)(m+1)}{2}}_{6} )\|u_t\|^{m+1}_{m+1}\notag\\
    & \le& \theta+\theta\|U_\eps\|^{k+1}_{3k+3}+C_\theta R^6\|u_t\|^{m+1}_{m+1},\\
    \left\vert\int_\Omega M_\eps(u)g(u_t) \td x\right\vert & \le& \theta+\theta \|U_\eps\|^{k+1}_{3k+3}+C_\theta(1+\|U_\eps\|^{\frac{k(m+1)}{m}}_{6})\|u_t\|^{m+1}_{m+1} +C\|U_\eps\|^{k}_{6}\notag \\
    & \le&\theta+\theta \|U_\eps\|^{k+1}_{3k+3}+C_\theta R^{6}\|u_t\|^{m+1}_{m+1} +CR^4,
  \end{eqnarray}
  hold for any given $\theta>0$.


  It remains to estimate the term involving nonlinearity $f$ on the right hand side of \eqref{eqh:new}. In fact, due to \eqref{fl-Flu} we have
  \begin{equation}
    \begin{aligned}
      M_\eps(u)f(u) & =\dfrac{|u|^{k-1}}{1+\eps|u|^{k-1}}f(u)u\ge \dfrac{|u|^{k-1}}{1+\eps|u|^{k-1}}(-\lambda u^2-C) \\
                    & \ge -\lambda |u|^{k+1}-C|u|^{k-1}\ge -\lambda_1 |u|^{k+1}-C,
    \end{aligned}
  \end{equation}
  thus
  \begin{equation}\label{hf}
    \int_\Omega M_\eps(u)f(u)\td x  \ge - \int_\Omega (\lambda_1|u|^{k+1}+C)\td x
    \ge -C\|u\|_6^{k+1}-C\ge -CR^4.
  \end{equation}

  Inserting \eqref{hd} -- \eqref{hphi}, and \eqref{hp} -- \eqref{hf} in \eqref{eqh:new} and setting $\theta$ small enough results in
  \begin{equation}\label{ste}
    \int_0^T \|U_\eps\|^{k+1}_{3k+3}\td t\le C\left( R^4T+R^4+CR^{6}\left(T+\int_0^T\|u_t\|^{m+1}_{m+1}\td t\right)\right)\le C(R^7+R^{6}T),
  \end{equation}
  where the constant $C$ is independent of $\eps$ and $u$.

  Finally, since $U_\eps\to u$ a.e. in $[0,T]\times \Omega$ as $\eps\to 0^{+}$, the desired assertion follows directly from Fatou's Lemma. The proof is complete.
\end{proof}

\begin{remark}\label{rem_fu}
  By H\"{o}lder's inequality, we have
  \begin{equation*}
    \int_\Omega |u|^{p(m+1)/m}\td x\le  \|u\|^{\alpha p(m+1)/m}_{3k+3}\|u\|^{(1-\alpha) p(m+1)/m}_{6},
  \end{equation*}
  where
  $$\alpha=\left(\frac16-\frac{m}{p(m+1)}\right)/\left(\frac16-\frac1{3k+3}\right).$$
  Since
  $$\alpha\frac{ p(m+1)}{m}=\frac{(p(m+1)-6m)(k+1)}{m(k-1)}\le k+1,$$
  for any weak solution $u(t)$ of problem \eqref{eq1}, by virtue of \eqref{F1}, we can deduce that
  $$f(u)\in L^{(m+1)/m}((0,T)\times\Omega).$$
  This fact is critical for uniqueness of the weak solution to be discussed in the next subsection.
\end{remark}

\subsection{Uniqueness of the weak solution}
With the a priori estimate \eqref{space-time}, we can now establish the uniqueness of the weak solution.

\begin{thm}[Uniqueness of the weak solution]\label{unique}
  Let Assumptions {\rm({\bf G})} and {\rm({\bf F})} hold, and $\phi\in L^2(\Omega)$. Suppose that $u(t, x)$ and $v(t, x)$ are two weak solutions of \eqref{eq1} on $[0, T]\times\Omega$, and set  $w(t, x)=u(t, x)-v(t, x)$ for $(t, x) \in [0, T]\times\Omega$. Then the following continuous dependence estimate on initial datum
  \begin{equation}\label{ieq_cd}
    \|(w(t),w_t(t))\|_{\H}\le C_{u,v}\|(w(0),w_t(0))\|_{\H}
  \end{equation}
  holds, where and $C_{u,v}$ depends on the $\H$ norm of the solutions $u$ and $v$.
\end{thm}

\begin{proof}
  First,  the difference $w(t, x)$ of the solutions $u(t, x)$ and $v(t, x)$ of \eqref{eq1} satisfies the equation
  \begin{align}\label{eqw}
    \begin{cases}w_{tt}+g(u_{t})-g(v_t)-\Delta w=f(v)-f(u), \,\,\,\text{~~in~~} \mathbb{R}_{+}\times\Omega, \\
      w(x,0)=u_0-v_0,~~ w_{t}(x,0)=u_1-v_1,  \quad \text{~~in~~} \Omega.
    \end{cases}
  \end{align}
  Due to  Lemma \ref{priori}, we have $u$, $v\in L^{k+1}(0,T;L^{3k+3}(\Omega))$. As stated in Remark \ref{rem_fu}, the space-time integrability in $L^{k+1}(0,T;L^{3k+3}(\Omega))$ ensures that $f(u), f(v)\in L^{(m+1)/m}((0,T)\times \Omega)$. This level of integrability is crucial because it allows us to handle the nonlinear terms $f(u)$ and $f(v)$ appropriately. Thus, we can take inner product of \eqref{eqw} with $w_t$, leading to
  \begin{equation}\label{eq23}
    \frac12 \frac{\td }{\td t}\left(\|w_t\|^2+ \|\nabla w\|^2\right)+\int_\Omega (g(u_{t})-g(v_t))w_t\td x=-\int_\Omega (f(u)-f(v))w_t\td x.
  \end{equation}
  Because $g$ is monotonically increasing,  $g^\prime(s)\ge 0$ and hence
  \begin{equation}\label{gu-gv}
    \int_\Omega (g(u_{t})-g(v_t))w_t\td x\ge 0.
  \end{equation}

  Next, rewrite the second term on the right hand side of \eqref{eq23} as
  \begin{equation}\label{fu-fv}
    \begin{aligned}
      -   \int_\Omega (f(u)-f(v))w_t\td x = & -\int_\Omega \int_0^1 f^{\prime}(\tau u+(1-\tau)v)\td \tau~ww_t\td x                             \\
      =                                     & - \frac12 \frac{\td }{\td t}\int_\Omega \int_0^1 f^{\prime}(\tau u+(1-\tau)v)\td \tau~w^2\td x   \\
                                            & + \int_\Omega\int_0^1 f^{\prime\prime}(\tau u+(1-\tau)v)(\tau u_t+(1-\tau)v_t)\td \tau~w^2\td x.
    \end{aligned}
  \end{equation}
  Notice that due to Assumption (F3), for any $\omega>0$, there exists $K_\omega > 0$ such that
  \begin{equation}
    \begin{aligned}
      \int_\Omega \int_0^1 f^{\prime}(\tau u+(1-\tau)v)\td \tau~w^2\td x & \ge -\int_\Omega (\omega u^4+\omega v^4+K_\omega)w^2\td x  \\
                                                                         & \ge -\omega(\|u\|_6^4+\|v\|_6^4)\|w\|_6^2-K_\omega\|w\|^2.
    \end{aligned}
  \end{equation}
  Setting
  $$E_\omega(t):=\frac12\left(\|w_t\|^2+\|\nabla w\|^2+\int_\Omega \int_0^1 f^{\prime}(\tau u+(1-\tau)v)\td \tau~w^2\td x+K_\omega\|w\|^2\right),$$
  and choosing $\omega>0$ small enough gives
  \begin{equation}
    E_\omega(t)\ge \frac14\left(\|w_t\|^2+\|\nabla w\|^2\right).
  \end{equation}
  On the other hand,  due to the growth condition \eqref{F1} and H\"older's inequality,  the second term  in the equality \eqref{fu-fv} satisfies
  \begin{equation}
    \begin{aligned}
          & \int_\Omega\int_0^1 f^{\prime\prime}(\tau u+(1-\tau)v)(\tau u_t+(1-\tau)v_t)\td \tau~w^2\td x                \\
      \le & C \int_\Omega (1+|u|^{p-2}+|v|^{p-2})(|u_t|+|v_t|)w^2\td x                                                   \\
      \le & C \left(\int_\Omega (1+|u|^{3(p-2)/2}+|v|^{3(p-2)/2})(|u_t|^{3/2}+|v_t|^{3/2})\td x\right)^{2/3}\|w\|^2_{6}  \\
      \le & C \left(\int_\Omega (1+|u|^{3(p-2)/2}+|v|^{3(p-2)/2})(|u_t|^{3/2}+|v_t|^{3/2})\td x\right)^{2/3} E_\omega(t) \\
      :=  & I(t)E_\omega(t).
    \end{aligned}
  \end{equation}
  Notice that the exponents $\lambda=3(p-2)/2$, $\mu=3/2$ and $q=3/2$ satisfy conditions of Lemma \ref{inter_ineq}, we have
  \begin{equation}\label{It}
    \begin{aligned}
      I(t)= & C\left(\int_\Omega (1+|u|^{3(p-2)/2}+|v|^{3(p-2)/2})(|u_t|^{3/2}+|v_t|^{3/2})\td x\right)^{2/3}                                                                     \\
      \le   & C_2\left( 1+\|u\|^{k+1}_{3k+3}+\|v\|^{k+1}_{3k+3} +\left(1+\|u\|^{3(m+1)}_{6}+\|v\|^{3(m+1)}_{6}\right)\left(\|u_t\|^{m+1}_{m+1}+\|v_t\|^{m+1}_{m+1}\right)\right).
    \end{aligned}
  \end{equation}
  Clearly, $I(\cdot)\in L^{1}(0,T)$.

  Substituting \eqref{gu-gv}-\eqref{It} into \eqref{eq23}, we obtain that the energy functional $E_\omega(t)$ satisfies
  \begin{equation}
    \begin{aligned}
      \frac{\td }{\td t}E_\omega(t) & \le \int_\Omega\int_0^1 f^{\prime\prime}(\tau u+(1-\tau)v)(\tau u_t+(1-\tau)v_t)\td \tau~w^2\td x+K_\omega\int_\Omega w w_t\td x \\
                                    & \le I(t)E_\omega(t)+C K_\omega E_\omega(t).
    \end{aligned}
  \end{equation}
  Then according to Gronwall's Lemma, we can deduce inequality \eqref{ieq_cd}.
\end{proof}

\subsection{Well-posedness of the strong solution}
To facilitate the construction of weak solutions and demonstrate the regularity of the global attractor in Section \ref{sec:attr}, we will establish the existence of strong solutions for system \eqref{eq1} in this subsection. Specifically, we will show that system \eqref{eq1} has a unique strong solution with nonlinear exponents residing in Region III, as illustrated in Figure \ref{fig:1}. To that end, we begin by recalling the definition of a strong solution as follows:
\begin{Def}[Strong solution \cite{chueshov08}]
  A function $u\in C([0,T];\V)\cap C^1([0,T];H^1_0(\Omega))$ possessing the properties $u(0)=u_0$ and $u_t(0)=u_1$ is said to be strong a solution to problem \eqref{eq1} on $[0,T]\times\Omega$, if and only if
  \begin{itemize}
    \item $u_t\in L^1(a,b;H^1_0(\Omega))$ and $u_{tt}\in L^1(a,b;L^2(\Omega))$ for any $0<a<b<T$;
    \item $-\Delta u+g(u_t)\in L^2(\Omega)$ for almost all $t\in [0,T]$;
    \item equation \eqref{eq1} is satisfied for almost all $t\in [0,T]$ and $x\in \Omega$.
  \end{itemize}
\end{Def}

\begin{thm}[Strong solution in Region III]\label{strong}
  Assume that $\phi\in L^2(\Omega)$ and let Assumptions {\rm({\bf G})} and {\rm({\bf F})} hold. Then for every $(u_0,u_1) \in \V$ and for any $T>0$, there exists a unique global strong solution $u(t)$ of equation \eqref{eq1} that satisfies the following energy estimate
  \begin{equation*}\label{energy_estimate3}
    \| (u(t),u_t(t))\|_{\V}\le \Gamma_T(\| (u_0,u_1)\|_{\V}+\|\phi\|),
  \end{equation*}
  where $\Gamma_T(\cdot)$ is a continuous increasing function depends on $T$.
\end{thm}
\begin{proof}
  The existence will be demonstrated using the Faedo-Galerkin method in five parts. In Part (i) the existence of solutions to a finite-dimensional approximating system is proved. In Part (ii) a space-time estimate for the Galerkin approximate solution is established. In Part (iii) a strong energy estimate in $\V$ is derived. In Part (iv), an $L^2(\Omega)$ estimate of $u^n_{tt}$ is established, and finally in Part (v) the existence of a strong solution is obained as the limit of solutions to the finite-dimensional approximating system.

    {(i) \it Solutions to the approximate problem}. Let $P_n$ be the orthogonal projection in $L^2(\Omega)$ onto the span of $\{e_1,e_2,...,e_n\}$ and define $Q_n := I-P_n$. Consider the following finite-dimensional approximate problem:
  \begin{equation}\label{eq_fdap}
    \begin{cases}
      u^n_{tt}+P_n g(u^n_t)-\Delta u^n+P_nf(u^n)=\phi_n=P_n \phi, \\
      u^n(0)=P_nu_0,u^n_t(0)=P_nu_1,
    \end{cases}
  \end{equation}
  which is, in fact,  a system of ordinary differential equations. Given that $g\in C^1(\R)$ and $f\in C^2(\R)$, by Picard's Existence Theorem, the system \eqref{eq_fdap} has a local classical solution  that can be expressed as $u^n(t)=\sum_{i=1}^{n}a_{in}(t)e_i.$

  Next, define the energy functional $E$ by
  $$E(u^n) :=\frac12\|u^n_t\|^2+\frac12\|\nabla u^n\|^2+\int_\Omega F(u^n)\td x-\int_\Omega \phi u^n\td x.$$
  Then multiplying \eqref{eq_fdap} by $u^n_t(t)$ and integrating over $\Omega$ gives
  \begin{equation*}\label{eq2.7}
    \frac{\td}{\td t}E(u^n(t))+\int_\Omega g(u^n_t) u^n_t\td x=0,
  \end{equation*}
  which can be integrated  from $0$ to $t$ to get
  \begin{equation}\label{eq2.8}
    E(u^n(t))+\int_0^t \int_\Omega g(u^n_t) u^n_t\td x\td s=E(u^n(0)).
  \end{equation}
  Notice that the growth condition \eqref{F1} and dissipative condition \eqref{F2} imply that
  $$C_1(\|u^n_t(t)\|^2+\|\nabla u^n(t)\|^2-C_2)\le  E(u^n(t))\le C_2\left(\|u^n_t(t)\|^2+\|\nabla u^n(t)\|^2+1\right)^{p/2}.$$
  This, along with \eqref{eq2.8} then lead to the energy estimate
  \begin{equation*}\label{eq2.9}
    \|u^n_t(t)\|^2+\|\nabla u^n(t)\|^2+\int_0^t\int_\Omega g(u^n_t) u^n_t\td x \td s \le C_3(E(u^n(0))+1),
  \end{equation*}
  which ensures that the local solution can be extended to any finite time interval.

  Furthermore, combining $|s|^{m+1}\le Cg(s)s$ for $|s|\ge 1$ with $|s|^{m+1}\le s^2$ for $|s|\le 1$ implies that $|s|^{m+1}\le Cg(s)s+s^2$, which yields
  \begin{equation}\label{u_t^m+1}
    \int_{0}^t\int_\Omega |u^n_t|^{m+1}\td x \td s\le  \int_{0}^t\int_\Omega \left( Cg(u^n_t) u^n_t+|u^n_t|^2\right)\td x \td s\le C(1+t)(E(u^n(0))+1).
  \end{equation}

  {(ii) \it Space-time estimate.}  In this part we establish an estimate similar to \eqref{space-time} for the Galerkin approximate solution. The proof is similar to that of Lemma \ref{space-time}, but there are some differences because the product $\la P_nf(u_n),P_n \psi(u^n) \ \ra $ does not usually coincide with $\la f(u_n), \psi(u^n)\ra$ for a nonlinear function $\psi$.

  Multiplying the approximate equation \eqref{eq_fdap} by $-\Delta u^n_t$ and integrating over $\Omega$. This yields
  \begin{equation}\label{Edu}
    \frac{\td }{\td t}E_{\V}(u^n(t))+\int_\Omega g^{\prime}(u^n_t)|\nabla u^n_t|^2\td x+\int_\Omega f^{\prime}(u^n)\nabla u^n\nabla u^n_t\td x=0,
  \end{equation}
  where $$E_{\V}(u^n(t)) :=\frac12\|\nabla u^n_t(t)\|^2+\frac12\|\Delta u^n(t)\|^2+\la \phi, \Delta u^n(t)\ra.$$

  For the nonlinearity term in \eqref{Edu}, we observe that
  \begin{equation}
    \left|\int_\Omega f_\eps^\prime(u^n)\nabla u^n\nabla u^n_t\td x\right| \le C\(\|\nabla u^n_t(t)\|^2+\|\Delta u^n(t)\|^2\)^{3}\le C\(E_{\V}(u^n(t))+\|\phi\|^2\)^3.
  \end{equation}
  Substituting into \eqref{Edu}, and taking into account $g^\prime(s) \ge0$ implies that
  \begin{equation}
    \frac{\td }{\td t}\(E_{\V}(u^n(t))+\|\phi\|^2\)\le C\(E_{\V}(u^n(t))+\|\phi\|^2\)^3.
  \end{equation}
  Integrating on $[0,t]$, one has
  \begin{equation}
    \(E_{\V}(u^n(t))+\|\phi\|^2\)^2\le \frac{\(E_{\V}(u^n(0))+\|\phi\|^2\)^2}{1-Ct\(E_{\V}(u^n(0))+\|\phi\|^2\)^2},
  \end{equation}
  Notice that $E_{\V}(u^n(0))$ is uniformly bounded, thus we assume that for some $t_1\in (0, T]$,
  \begin{equation*}
    \sup_n\sup_{0\le s \le t_1}\|(u^n(s), u^n_t(s))\|_{\V} <+\infty,
  \end{equation*}
  and set $S=\sup_n\sup_{0\le s \le t_1}\|(u^n(s), u^n_t(s))\|_{\V}+1$.

  Since the case $k=1$ is trivial, so we will focus on $k>1$, which implies that $m<5$.

  Define $M(s): = |s|^{k-1} s$. Multiplying the equation \eqref{eq_fdap} by the nonlinear test function $P_n M(u^n)$ and integrating over $[0, t_1]\times \Omega$, we obtain:
  \begin{equation}\label{Meq}
    \int_{0}^{t_1}\int_\Omega\left(M(u^n)u^n_{tt}+M^{\prime}(u^n)|\nabla u^n|^2+P_nM(u^n)(P_ng(u^n_t)+P_nf(u^n)-\phi_n)\right) \td x\td t=0.
  \end{equation}
  Noticing  that
  \begin{equation*}
    \int_\Omega P_nM(u^n)P_ng(u^n_t)\td x=\int_\Omega P_nM(u^n)g(u^n_t)\td x,
  \end{equation*}
  we have
  \begin{equation}\label{Mg}
    \begin{aligned}
      \int_\Omega (M(u^n)g(u^n_t)-P_nM(u^n)P_ng(u^n_t))\td x & = \int_\Omega Q_nM(u^n) g(u^n_t)\td x                \\
                                                             & \le \|Q_nM(u^n)\|_{H^{\al}}\|g(u^n_t)\|_{H^{-\al}} ,
    \end{aligned}
  \end{equation}
  where $\al=(m-1)/4$ and $Q_n = I - P_n$.

  Similar to the proof of Lemma \ref{space-time} define
  $$R: =\sup_n\left\{\sup_{0\le s\le T}(\|u^n(s)\|_6+\|u^n_t(s)\|)+\int_0^T\|u^n_t\|^{m+1}_{m+1}\td t\right\}+\|\phi\|+1.$$
  Then, according to the definitions of the norm of $H^{\al}$ and projection $Q_n$, we have
  \begin{equation}\label{new:7}
    \|Q_nM(u^n)\|_{H^{\al}}\le \lambda_n^{\frac{\al-1}2} \|Q_nM(u^n)\|_{H^1_0(\Omega)}\le C\lambda_n^{\frac{\al-1}2} \|M(u^n)\|_{H^1_0(\Omega)}\le C\lambda_n^{\frac{\al-1}2} S^{k-1}R,
  \end{equation}
  and  by fractional Sobolve embedding, the following holds
  \begin{equation}\label{new:8}
    \|g(u^n_t)\|_{H^{-\al}}\le C\|g(u^n_t)\|_{\frac{12}{m+5}}\le C(1+\|u_t\|_6^m)\le CS^m.
  \end{equation}
  Applying the inequalities \eqref{new:7} and \eqref{new:8} to the right side \eqref{Mg} results in
  \begin{equation}\label{new:9}
    \int_\Omega (M(u^n)g(u^n_t)-P_nM(u^n)P_ng(u^n_t))\td x\le C\lambda_n^{\frac{\al-1}2} S^{m+k-1}R\le C_4\lambda_n^{\frac{\al-1}2} S^{7}R.
  \end{equation}
  Following similar arguments it can be deduced that
  \begin{equation}\label{new:10}
    \begin{aligned}
          & \int_\Omega (M(u^n)f(u^n)-P_nM(u^n)P_nf(u^n))\td x
      =    \int_\Omega Q_nM(u^n)f(u^n)\td x                                                                                                                                                \\
      \le & \|Q_nM(u)\|_{H^{\al}}\|f(u^n)\|_{H^{-\al}}\le C\lambda_n^{\frac{\al-1}2} S^{k-1}R\|f(u^n)\|_{\frac{12}{m+5}}\le C\lambda_n^{\frac{\al-1}2} S^{k-1}R\|u^n\|^p_{\frac{12p}{m+5}} \\
      \le & C\lambda_n^{\frac{\al-1}2} S^{k-1}RS^p=C\lambda_n^{\frac{\al-1}2} S^{p+k-1}R \le C_4\lambda_n^{\frac{\al-1}2} S^{7}R.
    \end{aligned}
  \end{equation}
  Since $\lambda_n\to +\infty$ as $n\to \infty$,   there exists $N=N(S)\in \mathbb{N}$ such that $C_4\lambda_n^{\frac{\al-1}2} S^{7}R\le1$ for any $n\ge N(S)$. Hence it follows from \eqref{new:9} and \eqref{new:10} that
  \begin{eqnarray*}
    \int_\Omega P_nM(u^n)P_ng(u^n_t)\td x &\ge& \int_\Omega M(u^n)g(u^n_t)\td x-1, \quad n\ge N(S),\\
    \int_\Omega P_nM(u^n)P_nf(u^n)\td x &\ge& \int_\Omega M(u^n)f(u^n)\td x-1,\quad n\ge N(S),
  \end{eqnarray*}
  which can be applied to \eqref{Meq} to obtain that for any $n\ge N(S)$
  \begin{equation*}\label{}
    \int_{0}^{t_1}\int_\Omega\left(M(u^n)u^n_{tt}+M^{\prime}(u^n)|\nabla u^n|^2+M(u^n)(g(u^n_t)+f(u^n)-\phi_n)\right) \td x\td t\le 2t_1.
  \end{equation*}
  Following the procedure in the proof of Lemma \ref{space-time} with minor variations in technical details, we obtain the space-time estimate
  \begin{equation}\label{space_time_app}
    \begin{aligned}
      \int_{0}^{t_1} \|u^n\|^{k+1}_{3k+3}\td t & \le C\left(R^4+C R^{6}\left(t_1+\int_{0}^{t_1}\|u^n_t\|^{m+1}_{m+1}\td t\right)\right)+Ct_1 \\
                                               & \le C(R^7+R^{6}T), \qquad \forall n\ge N(S).
    \end{aligned}
  \end{equation}

  {(iii) \it Energy estimates in $\V$.}    To establish the strong energy estimate in $\V$, it is natural to multiply the approximate equation \eqref{eq_fdap} by $-\Delta u^n_t$ and integrate over $\Omega$. This yields
  \begin{equation}
    \frac{\td }{\td t}E_{\V}(u^n(t))+\int_\Omega g^{\prime}(u^n_t)|\nabla u^n_t|^2\td x+\int_\Omega f^{\prime}(u^n)\nabla u^n\nabla u^n_t\td x=0,
  \end{equation}
  where $$E_{\V}(u^n(t)) :=\frac12\|\nabla u^n_t(t)\|^2+\frac12\|\Delta u^n(t)\|^2+\la \phi, \Delta u^n(t)\ra.$$

  For the nonlinearity term in \eqref{Edu}, we observe that
  \begin{equation}\label{eq2.35}
    \int_\Omega f_\eps^\prime(u^n)\nabla u^n\nabla u^n_t\td x = \frac12 \frac{\td }{\td t}\int_\Omega f_\eps^\prime(u^n)|\nabla u^n|^2\td x-\int_\Omega f_\eps^{\prime\prime}(u^n)|\nabla u^n|^2 u^n_t\td x.
  \end{equation}

  Similar to the proof of Theorem \ref{unique}, for some small $\omega>0$, the energy functional defined by
  $$E_{\V,\omega}(u^n(t))=E_\V(u^n(t))+\frac12\int_\Omega f_\eps^\prime(u^n)|\nabla u^n|^2\td x+K_\omega\int_\Omega |\nabla u^n|^2\td x+2\|\phi\|^2$$
  satisfies
  \begin{equation}\label{Ev_omega}
    E_{\V,\omega}(u^n(t))\ge \frac14\|\nabla u^n_t(t)\|^2+\frac14\|\Delta u^n(t)\|^2.
  \end{equation}
  Due to H\"older's inequality, the second term on the right hand side of \eqref{eq2.35} satisfies
  \begin{equation} \label{new:11}
    \begin{aligned}
      \left|\int_\Omega f_\eps^{\prime\prime}(u^n)|\nabla u^n|^2 u^n_t\td x\right|
       & \le C \int_\Omega (1+|u^n|^{p-2})|\nabla u^n|^2~|u^n_t|\td x                                                       \\
       & \le C \left(\|u^n_t\|+\left(\int_\Omega |u^n|^{(3p-6)/2}|u^n_t|^{3/2}\td x\right)^{2/3}\right)\|\nabla u^n\|^2_{6} \\
       & \le C \left(\|u^n_t\|+\left(\int_\Omega |u^n|^{(3p-6)/2}|u^n_t|^{3/2}\td x\right)^{2/3}\right)\|\Delta u^n\|^2.
    \end{aligned}
  \end{equation}
  Notice that the exponents $\lambda=(3p-6)/2$, $\mu=3/2$ and $q=3/2$ satisfy
  $$\frac{3p-6}{6(k+1)}+\frac{3}{2(m+1)}=1-\frac{(k-1)(m-1)}{2(k+1)(m+1)} \le 1,$$
  and
  $$\frac{3p-6-(k-1)}{2(k+5)}+\frac{3}{2(m+1)}=1-\frac{3}{2}\left(\frac{m}{m+1}-\frac{p}{k+5}\right) \le 1. $$
  Lemma \ref{inter_ineq} can then be applied to the term $\left(\int_\Omega |u^n|^{(3p-6)/2}|u^n_t|^{3/2}\td x\right)^{2/3}$ in \eqref{new:11} to give
  \begin{equation*}
    \left(\int_\Omega |u^n|^{(3p-6)/2}|u^n_t|^{3/2}\td x\right)^{2/3}\le 1+\|u^n\|^{k+1}_{3k+3}+C\left(1+\|u^n\|^{(p-2)(m+1)}_{6}\right)\|u^n_t\|^{m+1}_{m+1}.
  \end{equation*}
  Consequently it follows by \eqref{new:11} that
  \begin{equation}\label{fpp_du_ut}
    \begin{aligned}
          & \int_\Omega f_\eps^{\prime\prime}(u^n)|\nabla u^n|^2 u^n_t\td x                                                                \\
      \le & \left(\|u^n_t\|+1+\|u^n\|^{k+1}_{3k+3}+ C_5\left(1+\|u^n\|^{(p-2)(m+1)}_{6}\right)\|u^n_t\|^{m+1}_{m+1}\right)\|\Delta u^n\|^2 \\
      =   & H(u^n(t))E_{\V,\omega}(u^n(t)),
    \end{aligned}
  \end{equation}
  where
  \begin{equation*}
    H(u^n(t)) : =\|u^n_t\|+1+\|u^n\|^{k+1}_{3k+3}+C_5\left(1+\|u^n\|^{(p-2)(m+1)}_{6}\right)\|u^n_t\|^{m+1}_{m+1}.
  \end{equation*}
  Substituting \eqref{eq2.35} and \eqref{fpp_du_ut} into \eqref{Edu} and taking into account $g^{\prime}(s)\ge 0$ we obtain
  \begin{equation*}
    \begin{aligned}
      \frac{\td }{\td t}E_{\V,\omega}(u^n(t)) & \le H(u^n(t))E_{\V,\omega}(u^n(t))+K_\omega\la \nabla u^n_t, \nabla u^n\ra \\
                                              & \le C_6(H(u^n(t))+K_\omega)E_{\V,\omega}(u^n(t)),
    \end{aligned}
  \end{equation*}
  from which it followed by Gronwall's Lemma that
  \begin{equation*}
    E_{\V,\omega}(u^n(t))\le E_{\V,\omega}(u^n(0))\exp\left(\int_{0}^{t_1} C_6(H(u^n(t))+K_\omega)\td t\right).
  \end{equation*}
  Moreover, from  the space-time estimate \eqref{space_time_app} we have
  \begin{equation*}
    \begin{aligned}
      \int_{0}^{t_1} H(u^n(t))\td t & \le  (R+1)T+C(R^6+R^5T)+C(1+R^{(p-2)(m+1)})R \\
                                    & \le C_7(R^{19}+R^5T), ~~\forall n\ge N(S),
    \end{aligned}
  \end{equation*}
  and hence
  \begin{equation}\label{Ev_omega_R}
    E_{\V,\omega}(u^n(t))\le E_{\V,\omega}(u^n(0))e^{C_7(R^{19}+R^5T)+C_6K_\omega T}, ~~\text{for any  }t\le t_1 \text{~and~} n\ge N(S).
  \end{equation}

  On the other hand, due to the condition \eqref{F2} and Sobolev embedding inequality,
  \begin{equation*}
    \int_\Omega f_\eps^\prime(u^n)|\nabla u^n|^2\td x\le C(1+\|u^n\|_6^4)\|\nabla u^n\|_6^2\le C_8(1+\|\nabla u^n\|^4)\|\Delta u^n\|^2,
  \end{equation*}
  and thus
  \begin{equation*}
    E_{\V,\omega}(u^n(0))\le C_9\left(\|\nabla u_1\|^2+(1+\|\nabla u_0\|^4)\|\Delta u_0\|^2+\|\phi\|^2\right).
  \end{equation*}
  Now we define
  $$\bar{S}^2=4C_9\left(\|\nabla u_1\|^2+(1+\|\nabla u_0\|^4)\|\Delta u_0\|^2+\|\phi\|^2\right) e^{C_7(R^{19}+R^5T)+C_6K_\omega T}. $$
  Then it follows from \eqref{Ev_omega} and \eqref{Ev_omega_R} that
  \begin{equation*}
    \|(u^n(t), u^n_t(t))\|_{\V}\le \bar{S}, ~~\text{for any  }t\le t_1 \text{~and~} n\ge N(S).
  \end{equation*}
  Then
  \begin{equation}
    \limsup_{n\to \infty}\sup_{0\le s \le t_1}\|(u^n(s), u^n_t(s))\|_{\V}\le \bar{S}.
  \end{equation}
  From which we can show that
  \begin{equation}\label{new:12}
    \limsup_{n\to \infty}\sup_{0\le t\le T}\|(u^n(t), u^n_t(t))\|_{\V}\le  \bar{S}.
  \end{equation}
  In fact, suppose (for contradiction) that there exists a $T_0\in (0,T)$ such that
  \begin{equation}\label{new:13}
    \limsup_{n\to \infty}\lim_{t\to T_0^{-}}\|(u^n(t), u^n_t(t))\|_{\V}=+\infty \text{~~and~}\limsup_{n\to \infty}\sup_{0\le t\le T_0-\eps}\|(u^n(t), u^n_t(t))\|_{\V}<\infty
  \end{equation}
  for any small $\eps>0$. Then by Part (ii) and \eqref{new:12}, we have
  $$\limsup_{n\to \infty}\sup_{0\le t\le T_0-\eps}\|(u^n(t), u^n_t(t))\|_{\V}\le \bar{S},~~\text{for any  } \varepsilon>0,$$
  which contracts with \eqref{new:13} after passing the limit $\eps\to 0^{+}$.

  {(iv) \it $L^2(\Omega)$ estimate of $u^n_{tt}$.}  First, note that it follows directly from the equation in \eqref{eq_fdap} that
  \begin{equation*}\label{}
    \begin{aligned}
          & \|u^n_{tt}(t)\|^2+\|P_n g(u^n_t(t))\|^2+2\int_\Omega u^n_{tt}(t)g(u^n_t(t))\td x \\
      \le & 3\|\Delta u^n(t)\|^2+3\|P_nf(u^n(t))\|^2+3\|\phi_n\|^2.
    \end{aligned}
  \end{equation*}
  Since, by \eqref{new:12}, $\|\Delta u^n(t)\|\le S$ and $\|P_nf(u^n(t))\|\le S^5$, integrating above inequality over $[0, T]$ gives
  \begin{equation*}
    \begin{aligned}
      \int_0^T \int_\Omega |u^n_{tt}(t)|^2\td x \td t
       & \le  -2\int_0^T\int_\Omega u^n_{tt}(t)g(u^n_t(t))\td x\td t+\Gamma_T(\| (u_0,u_1)\|_{\V}+\|\phi\|)                   \\
       & =  2\int_\Omega \left(G(P_nu_1)-G(u^n_t(T))\right)\td x+\Gamma_T(\| (u_0,u_1)\|_{\V}+\|\phi\|),~~ \forall n\ge N(S),
    \end{aligned}
  \end{equation*}
  where $G(s)=\int_0^s g(\tau)\td \tau$ is the primitive function of $g$, and $\Gamma_T$ is a continuous increasing function of $T$. In addition, due to
  the growth condition \eqref{G1}, one gets $0\le G(s)\le C(1+|s|^{6})$, and thus
  \begin{equation*}\label{untt}
    \int_0^T \int_\Omega |u^n_{tt}(t)|^2\td x \td t \le \Gamma_T(\| (u_0,u_1)\|_{\V}+\|\phi\|),\text{~~for any } n\ge N(S).
  \end{equation*}

  {(v) \it Passing to the limit.}    As $n\to \infty$, we can without loss of generality,  assume the following convergence:
  \begin{equation*}
    \begin{aligned}
       & u^n\to u \text{~~ weakly star in~~}L^{\infty}(0,T;H^2(\Omega)\cap H^1_0(\Omega)),                                             \\
       & u^n_t \to u_t \text{~~ weakly star in~~} L^{\infty}(0,T;H^1_0(\Omega)),                                                       \\
       & u^n_{tt} \to u_{tt} \text{~~ weakly in~~} L^{2}((0,T)\times\Omega),                                                           \\
       & (g^{\prime}(u^n_t))^{1/2}\partial_{x_i} u^n_t\to \xi_i \text{~~ weakly in~~} L^{2}([0,T]\times\Omega) \text{~for } i=1, 2 ,3.
    \end{aligned}
  \end{equation*}
  Then according to the Aubin-Lions Lemma\cite{simon87}, we also have strong convergence:
  \begin{equation}
    \begin{aligned}
       & \nabla u^n \to \nabla u \text{~~ in~~} C(0,T;L^5(\Omega)), \\
       & u^n\to u \text{~~ in~~}C([0,T]\times \Omega),              \\
       & u^n_t \to u_t \text{~~ in~~} C(0,T;L^5(\Omega)),           \\
       & u^n_t \to u_t \text{~~ a.e. in~~} [0,T]\times\Omega.
    \end{aligned}
  \end{equation}
  We observe that $(g^{\prime}(u^n_t))^{1/2}\partial_{x_i} u^n_t=\partial_{x_i}g^*(u^n_t)$ for some function $g^*$. Then due to the almost everywhere convergence $u^n_t\to u_t$  we have
  $$\partial_{x_i}g^*(u^n_t)\to \partial_{x_i}g^*(u_t) \text{~~ weakly in~~} L^{2}([0,T]\times\Omega) \text{~for } i=1, 2 ,3.$$
  Combining this with weak convergence $(g^{\prime}(u^n_t))^{1/2}\partial_{x_i}u^n_t\to \xi_i$, we can deduce that $\xi_i=\partial_{x_i}g^*(u_t)=(g^{\prime}(u_t))^{1/2}\partial_{x_i}u_t$.

  Finally, by taking the limit $n\to \infty$ in the approximate equation \eqref{eq_fdap}, it can be seen that the limit function $u(t)$ is in fact a strong solution of equation \eqref{eq1}. The proof is complete.
\end{proof}

\subsection{Existence of the weak solution}

In this section, we establish the existence of  weak solutions by approximating strong solutions proved above.

\begin{thm}[Existence of the weak solution]\label{weak}
  Assume that $\phi\in L^2(\Omega)$ and Assumptions {\rm({\bf G})} and {\rm({\bf F})} hold. Then, for every $(u_0,u_1) \in \H$, there exists a global solution $u(\cdot)$ of problem \eqref{eq1} with the energy estimate
  \begin{equation}\label{weak_energy_est}
    \| (u(t),u_t(t))\|_{\H}+\int_{0}^{t}\int_{\Omega}g(u_{t})u_{t}\td x\td t\le Q(\| (u_0,u_1)\|_{\H}+\|\phi\|),
  \end{equation}
  holding for any $t\ge 0$, where $Q$ is a continuous increasing function that is independent of $u$.
\end{thm}
\begin{proof}
  Consider a sequence $(u^n_0,u^n_1)\in \V$ such that $u^n_0 \to u_0$ in $H^1_0(\Omega)$ and $u^n_1 \to u_1$ in $L^2(\Omega)$. We can obtain a sequence of strong solutions $u^n(t)$ with initial data $(u^n_0,u^n_1)$ that satisfy the following estimate
  \begin{equation*}
    \| (u^n(t),u^n_t(t))\|_{\H}+\|u^n\|_{L^{k+1}(t,t+1;L^{3k+3}(\Omega))}\le  Q(\| (u_0,u_1)\|_{\H}+\|\phi\|).
  \end{equation*}
  By Lemma \ref{priori},  we also have the a priori space-time estimate
  \begin{equation}
    \|u^n\|_{L^{k+1}(t,t+1;L^{3k+3}(\Omega))}\le Q(\| (u_0,u_1)\|_{\H}+\|\phi\|).
  \end{equation}
  Without loss of generality, we assume that
  \begin{align*}
     & u^n\to u \text{ weakly star in } L^\infty(0,T;H^1_0(\Omega)),              \\
     & u^n\to u \text{ in } L^5((0,T)\times\Omega) \text{ and almost everywhere,} \\
     & u^n_t\to u_t \text{ weakly in } L^{m+1}((0,T)\times\Omega),                \\
     & g(u^n_t)\to \tilde{g} \text{ weakly in } L^{(m+1)/m}((0,T)\times\Omega).
  \end{align*}
  Passing to the limit as $n\to \infty$ in the definition of weak solution, then
  $$\phi=u^n_{tt}-\Delta u^n+f_{\eps_n}(u^n)+g(u^n_t)\to u_{tt}-\Delta u+f(u)+\tilde{g}$$
  weakly-star in $L^\infty(0,T;H^{-1}(\Omega))+L^{(m+1)/m}((0,T)\times\Omega)$.

  Since $u_t\in L^{m+1}((0,T)\times\Omega)$, we can multiply equation $u_{tt}-\Delta u=-f(u)-\tilde{g}+\phi\in L^{(m+1)/m}((0,T)\times\Omega)$ by $u_t$ and integrate over $(0,t)\times\Omega$, yielding the energy equality
  \begin{equation}
    E(u(t))+\int_0^t\int_\Omega \tilde{g}u_s\td x\td s=E(u(0)).
  \end{equation}
  For the approximate solution $u^n(t)$, the following energy equality also holds:
  \begin{equation}
    E(u^n(t))+\int_0^t\int_\Omega g(u^n_s)u^n_s\td x\td s=E(u^n(0)).
  \end{equation}
  Along with
  $$E(u(t))\le \liminf_{n\to \infty}E(u^n(t))\text{~~and~~} E(u(0))= \lim_{n\to \infty}E(u^n(0)),$$
  we obtain
  \begin{equation}
    \int_0^t\int_\Omega \tilde{g} u_s\td x\td s\ge \liminf_{n\to \infty}\int_0^t\int_\Omega g(u^n_s)u^n_s\td x\td s.
  \end{equation}
  This is a sufficient condition, according to the monotonicity argument of Lions\cite{Lions65}, to guarantee $\tilde{g}=g(u_t)$. Therefore, the limit function $u(t)$ is a weak solution with desired estimate. The proof is complete.
\end{proof}
We are now ready to define the solution semigroup $S(t):\H\rightarrow\H $ associated with equation \eqref{eq1} as
\begin{equation}\label{def:S}
  S(t) (u_0,u_1)=  (u(t),u_t(t)),
\end{equation}
where $u(t)$ is the unique weak solution of problem \eqref{eq1}. According to Theorem \ref{unique} and Theorem \ref{weak}, this semigroup is well-defined.
\begin{remark}\label{re2}
  In the higher dimension $N\ge 4$, the situations will be more complex, we would just state the result below and omit the proof.

  Assume that $\phi\in L^2(\Omega)$ and Assumptions {\rm({\bf G})} and {\rm({\bf F})} hold, however the restriction of $p$ in \eqref{F1} be replaced by \begin{equation}
    p\le \begin{cases}
      \min\{\frac{N+2}{N-2}, \frac{Nm}{N-2}\}, & N=4,                        \\
      \frac{N+2}{N-2},                         & m> \frac{N}{2}-1, N\ge 5,   \\
      \frac{Nm}{N-2},                          & m\le \frac{N}{2}-1, N\ge 5.
    \end{cases}
  \end{equation}
  Then, for every $(u_0,u_1) \in \H$, the weak solution $u(\cdot)$ of problem \eqref{eq1} is well-posedness and satisfies the following a priori space-time estimate
  \begin{equation}
    \int_0^T \|u\|^{k+1}_{\frac{n(k+1)}{n-2}}\td t\le Q(R,T),
  \end{equation}
  where $k=\min\{\frac{n+2}{(n-2)m},\frac{nm}{n-2}-\frac{2}{n-2}\}$,
  $$R=\sup_{0\le s \le T}(\|\nabla u(s)\|+\|u_t(s)\|)+\int_0^T\|u_t\|^{m+1}_{m+1}\td t+\|\phi\|+1,$$
  and $Q$ is a continuous increasing function.
\end{remark}

\section{Global attractor}\label{sec:attr}

In this section, we prove the existence of global attractor for problem \eqref{eq1}.
For  the reader's convenience, we now recall the definition of the global attractor; see \cite{babin92,temam97} for more details.
\begin{Def}
  Let $\{S(t)\}_{t\geq 0}$ be a semigroup on a metric space $(X,d)$.
  A subset $\mathscr{A}$ of $X$ is called a global attractor for the
  semigroup, if $\mathscr{A}$ is
  compact and enjoys the following properties:\\
  (1) $\mathscr{A}$ is invariant , i.e. $S(t)\mathscr{A}=\mathscr{A}$,~$\forall t\geq 0$;\\
  (2) $\mathscr{A}$ attracts all bounded sets of $X$, that is, for any bounded
  subset $B$ of $X$,
  $$ dist(S(t)B,\mathscr{A})\rightarrow 0,~~~as ~~t\rightarrow 0 ,$$
  where $dist(B,A)$ is the semidistance of two sets B and A:$$dist(B,A)=\sup \limits_{x\in B}\inf\limits_{y\in A}d(x,y).$$
\end{Def}
According to the abstract attractor existence theorem (see \cite{temam97}, for instance), the existence of the global attractor can be guaranteed provided the semigroup $\{S(t)\}_{t \ge 0}$ is continuous, dissipative, and asymptotically compact in the phase space $\H$. The continuity of $S(t)$ has been proven in Lemma \ref{unique}; therefore, it remains to verify the dissipativity and asymptotic compactness of the semigroup defined in \eqref{def:S}. In particular,  we show that
$\{S(t)\}_{t \ge 0}$ is dissipative  by contructing a bounded absorbing set in Lemma \ref{le4.1} below and prove the asymptotic compactness of $\{S(t)\}_{t \ge 0}$ in Lemma \ref{le4.6} below.

\begin{lemma}[Bounded absorbing set]\label{le4.1}
  Let the assumptions of Theorem \ref{weak} hold. Then the semigroup $\{S(t)\}_{t \geq 0}$ generated by the weak solution of problem \eqref{eq1} has a bounded absorbing set $\mathscr{B}_0$ in $\mathcal{H}$. That is, for any bounded subset $B \subset \mathcal{H}$, there exists $T_0$ depending on the $\mathcal{H}$-bounds of $B$ such that $S(t)B \subset \mathscr{B}_0$ for all $t \geq T_0$.
\end{lemma}

\begin{proof}
  Multiplying equation \eqref{eq1} by $u_t + \alpha u$ ($\alpha > 0$), we obtain the following equality:
  \begin{equation}\label{eq4.2}
    \begin{aligned}
       & \frac{\td}{\td t} E_\alpha(u) + \int_\Omega g(u_t) u_t \, \td x + \alpha \int_\Omega g(u_t) u \, \td x - \alpha \|u_t\|^2 + \alpha \|\nabla u\|^2 \\
       & \quad + \alpha \int_{\Omega} f(u) u \, \td x - \alpha \langle \phi, u \rangle = 0,
    \end{aligned}
  \end{equation}
  where
  \[
    E_\alpha(u) = E(u) + \alpha \langle u, u_t \rangle.
  \]

  Using H\"older's inequality and \eqref{g_delta}, we have
  \begin{equation*}
    \begin{aligned}
      \int_\Omega |g(u_t) u| \, \td x & \leq \left( \int_\Omega |g(u_t)|^{\frac{m+1}{m}} \, \td x \right)^{\frac{m}{m+1}} \|u\|_{m+1}                        \\
                                      & \leq C \left( \int_\Omega \left( c(\delta) g(u_t) u_t + \delta \right) \, \td x \right)^{\frac{m}{m+1}} \|\nabla u\| \\
                                      & \leq C \|\nabla u\|^{\frac{m-1}{m}} \int_\Omega \left( c(\delta) g(u_t) u_t + \delta \right) \, \td x
      + \frac{\lambda_1 - \lambda}{6\lambda_1} \|\nabla u\|^2                                                                                                \\
                                      & \leq C_1 \|\nabla u\|^{\frac{m-1}{m}} \int_\Omega c(\delta) g(u_t) u_t \, \td x + \delta
      + \left( C_1 \delta + \frac{\lambda_1 - \lambda}{6\lambda_1} \right) \|\nabla u\|^2                                                                    \\
                                      & \leq C_1 \|\nabla u\|^{\frac{m-1}{m}} \int_\Omega c(\delta) g(u_t) u_t \, \td x + \delta
      + \frac{\lambda_1 - \lambda}{3\lambda_1} \|\nabla u\|^2,
    \end{aligned}
  \end{equation*}
  where $0 < \delta \leq \frac{\lambda_1 - \lambda}{6\lambda_1 C_1}$. It then follows that
  \begin{equation*}
    \begin{aligned}
           & \int_\Omega g(u_t) u_t \, \td x + \alpha \int_\Omega g(u_t) u \, \td x - \alpha \|u_t\|^2                                           \\
      \geq & \int_\Omega g(u_t) u_t \, \td x - \alpha C_1 \|\nabla u\|^{\frac{m-1}{m}} \int_\Omega c(\delta) g(u_t) u_t \, \td x - \alpha \delta
      - \frac{\alpha (\lambda_1 - \lambda)}{3\lambda_1} \|\nabla u\|^2 - \alpha \|u_t\|^2                                                        \\
      \geq & (1 - \alpha C_1 c(\delta)) \|\nabla u\|^{\frac{m-1}{m}} \int_\Omega g(u_t) u_t \, \td x - \alpha \delta
      - \frac{\alpha (\lambda_1 - \lambda)}{3\lambda_1} \|\nabla u\|^2 - \alpha \|u_t\|^2.
    \end{aligned}
  \end{equation*}
  By choosing $\alpha = \alpha(\|B\|_{\mathcal{H}})$ small enough such that $\alpha c(\delta) (2 + C_1 \|\nabla u\|^{1 - 1/m}) \leq 1$, we obtain
  \begin{equation*}
    1 - \alpha C_1 c(\delta) \|u\|^{1 - 1/m}_6 \geq 2 \alpha c(\delta),
  \end{equation*}
  and consequently,
  \begin{equation}\label{eq4.6}
    \begin{aligned}
           & \int_\Omega g(u_t) u_t \, \td x + \alpha \int_\Omega g(u_t) u \, \td x - \alpha \|u_t\|^2           \\
      \geq & 2 \alpha c(\delta) \int_\Omega g(u_t) u_t \, \td x - \alpha \delta
      - \frac{\alpha (\lambda_1 - \lambda)}{3\lambda_1} \|\nabla u\|^2 - \alpha \|u_t\|^2                        \\
      \geq & 2 \alpha (\|u_t\|^2 - \delta) - \alpha \delta
      - \frac{\alpha (\lambda_1 - \lambda)}{3\lambda_1} \|\nabla u\|^2 - \alpha \|u_t\|^2                        \\
      =    & \alpha \|u_t\|^2 - 3\alpha \delta - \frac{\alpha (\lambda_1 - \lambda)}{3\lambda_1} \|\nabla u\|^2.
    \end{aligned}
  \end{equation}

  Using assumption \eqref{F2}, we obtain
  \begin{equation}\label{eq4.7}
    \int_\Omega f(u) u \, \td x - \langle \phi, u \rangle \geq -C_2 - \lambda \|u\|^2 - \frac{\lambda_1 - \lambda}{3\lambda_1} \|\nabla u\|^2 - C_3 \|\phi\|^2.
  \end{equation}
  Combining inequality \eqref{eq4.6} with \eqref{eq4.7}, we have
  \begin{equation}\label{eq4.8}
    \begin{aligned}
       & \int_\Omega g(u_t) u_t \, \td x + \alpha \int_\Omega g(u_t) u \, \td x - \alpha \|u_t\|^2 + \alpha \|\nabla u\|^2 + \alpha \int_\Omega f(u) u \, \td x - \alpha \langle \phi, u \rangle \\
       & \geq \alpha \left( \|u_t\|^2 + \frac{(\lambda_1 - \lambda)}{3\lambda_1} \|\nabla u\|^2 - C_2 - 3\delta - C_3 \|\phi\|^2 \right)                                                         \\
       & \geq \alpha \left( \frac{(\lambda_1 - \lambda)}{3\lambda_1} \|u_t\|^2 + \frac{(\lambda_1 - \lambda)}{3\lambda_1} \|\nabla u\|^2 - C_2 - 3\delta - C_3 \|\phi\|^2 \right)                \\
       & := \alpha J(t).
    \end{aligned}
  \end{equation}
  Substituting \eqref{eq4.8} into \eqref{eq4.2} and integrating with respect to time on $[0, t]$ gives
  \[
    E_\alpha(u(t)) - E_\alpha(u(0)) \leq -\alpha \int_0^t J(s) \, \td s.
  \]
  From the energy estimate, we know that  $E_\alpha(u(0))$ is bounded whenever the initial data set $B$ is fixed. Then for a given initial set $B$, there exists a time $T_B = T(\|B\|_{\mathcal{H}}, M)$ such that, for any $(u_0, u_1) \in B$, there exists $t_0 = t_0(u_0, u_1) \in [0, T_B]$ with $J(t_0) \leq 1$.

  Let
  $$M = \frac{3\lambda_1 (C_2 + 3\delta + C_3 \|\phi\|^2)}{\lambda_1 - \lambda},$$
  then \eqref{weak_energy_est} implies $\|(u(t), u_t(t))\|_{\mathcal{H}} \leq Q(M, \|\phi\|)$ for all $t \geq T_B$. Thus, we conclude that the set $\mathscr{B}_0 = \{(u, v) \in \mathcal{H} \mid \|(u, v)\|_{\mathcal{H}} \leq Q(M, \|\phi\|)\}$ is a bounded absorbing set. The proof is complete.
\end{proof}


We now prove the asymptotic compactness using the energy medthod(see, e.g., \cite{ball04,mrw98}).
\begin{lemma}[Asymptotic compactness]\label{le4.6}
  Let the assumptions of Theorem \ref{weak} hold. Then the semigroup associated with \eqref{eq1} is asymptotically compact, that is, for every sequence $\{z_n\}^{\infty}_{n=1}\subset \mathscr{B}_0$, and every sequence of times $t_n\to \infty$, there exists a subsequence $\{n_k\}$ such that
  $$\{S(t_{n_k})z_{n_k}\}^{\infty}_{=1}\rightarrow z \text{~~strongly in~~} \H.$$
\end{lemma}
\begin{proof}
  First, define
  \begin{equation*}
    (v_n(t),\p_t v_n(t))=
    \begin{cases}
      S(t+t_n)z_n, & t\ge -t_n \\
      0,           & t<-t_n.
    \end{cases}
  \end{equation*}
  Since $\{(v_n(t),\p_t v_n(t))\}$ is bounded in $L^\infty(\R;\H)$ and $\{\p_t v_n(t)\}$ is bounded in $L^{m+1}(\R\times \Omega)$, by Banach-Alaoglu theorem and Eberlein-Smulian theorem, we can assume without loss of generality that there exist 
  $\tilde{g}\in L^{\frac{m+1}m}(\R\times \Omega)$ and $y=(v,v_t)\in L^\infty(\R;\H)$, such that
  \begin{equation*}\label{eq4.10}
    \begin{aligned}
      (v_n,\p_t v_n) & \wto y\text{~weakly star in~}L^\infty(\R;\H),                        \\
      \p_t v_n       & \wto v_t \text{ weakly~in~} L^{m+1}(\R\times \Omega),                \\
      g(\p_t v_n)    & \wto \tilde{g} \text{ weakly in } L^{\frac{m+1}{m}}(\R\times\Omega).
    \end{aligned}
  \end{equation*}

  In the next step, we will show that $\tilde{g}=g(v_t)$. Rewriting the equality \eqref{eq4.2} as follows:
  \begin{equation}\label{eq4.11}
    \frac{\td }{\td t}E_\al(u(t))+ \al E_\al(u(t))+G_\al(u)+N_\al(u)+\Phi_\al(u)=0,
  \end{equation}
  where the modified energy terms are defined by:
  \begin{align*}
     & E_\al(u(t))=E(u(t))+\al\langle u(t),u_t(t)\rangle,                                  \\
     & G_\al(u)=\int_\Omega \left(g(u_t)u_t-\frac{3\al}2 u_t^2+\alpha g(u_t)u\right)\td x, \\
     & N_\al(u)=\frac\al2\|\nabla u\|^2-\al^2\langle u,u_t\rangle,                         \\
     & \Phi_\al(u)=\al\int_{\Omega} (f(u)u-F(u))\td x.
  \end{align*}
  Replacing $u$ with the sequence $v_n$ in \eqref{eq4.11},   multiplying by $e^{\al t}$, and integrating over $[-t_n,0]$ gives
  \begin{equation}\label{eq4.12}
    E_\al(v_n(0))+\int_{-t_n}^0e^{\al s}\left(G_\al(v_n(s))+N_\al(v_n(s))+\Phi_\al(v_n(s))\right)\td s=E_\al(z_n)e^{-\al t_n}.
  \end{equation}
  In order to pass the limit $n\to \infty$, we will handle each term in \eqref{eq4.12} individually.

  Firstly, the Aubin-Lions lemma implies that
  $$v_n\to v \text{~in~} C_{loc}(\R;L^5(\Omega)).$$
  Thus, $v_n\to v$ almost everywhere in $(0,\infty)\times \Omega$. Since $F(v_n)$ is bounded from below and weakly lower semi-continuous,   $E_\alpha$ and $N_\al$ are weakly lower semi-continuous on $\H$.
  From the weak convergence
  $$(v_n(0),\p_tv_n(0))=S(t_n)z_n\wto y(0)=(v(0),\p_tv(0))\text{~in~}\H,$$
  we obtain
  \begin{equation}\label{eq4.13}
    \varliminf_{n\to \infty}E_\al(v_n(0))\ge E_\al(v(0)),~~ \varliminf_{n\to \infty}N_\al(v_n(0))\ge N_\al(v(0)).
  \end{equation}

  For the term involving $\Phi_\al(v_n)$ in \eqref{eq4.12}, notice that
  due to \eqref{fs-F}, $f(s)s-F(s)+\omega s^6+K_\omega s^2$ is bounded from below for every $\omega>0$. Thus Fatou's Lemma implies that
  \begin{equation*}
    \begin{aligned}
       & \varliminf_{n\to\infty}\int_{-t_n}^0e^{\al s} \left(\Phi_\al(v_n)+\int_\Omega(\omega v_n^6+K_\omega v_n^2)\td x\right) \td s \\
       & \ge  \int_{-\infty}^0e^{\al s} \left(\Phi_\al(v)+\int_\Omega(\omega v^6+K_\omega v^2)\td x\right) \td s.
    \end{aligned}
  \end{equation*}
  Meanwhile, $\|v_n\|^2 \to \|v\|^2$ in $C_{loc}(\R)$ yields
  \begin{equation*}
    \lim_{n\to\infty}\int_{-t_n}^0e^{\al s}\int_\Omega v_n^2\td x \td s= \int_{-\infty}^0e^{\al s}\int_\Omega v^2\td x \td s.
  \end{equation*}
  Therefore,
  \begin{equation}\label{new:20}
    \begin{aligned}
      \varliminf_{n\to\infty}\int_{-t_n}^0e^{\al s} \Phi_\al(v_n) \td s+\omega\varlimsup_{n\to\infty}\int_{-t_n}^0e^{\al s} \int_\Omega (v_n^6-v^6)\td x \td s
      \ge \int_{-\infty}^0e^{\al s} \Phi_\al(v)\td s.
    \end{aligned}
  \end{equation}
  Taking the limit  of \eqref{new:20} as $\omega\to0$ results in
  \begin{equation}\label{eq4.9}
    \varliminf_{n\to\infty}\int_{-t_n}^0e^{\al s} \Phi_\al(v_n) \td s\ge \int_{-\infty}^0e^{\al s} \Phi_\al(v)\td s.
  \end{equation}
  Passing to the limit as $n\to \infty$ in equality \eqref{eq4.12} and taking into account \eqref{eq4.13} and \eqref{eq4.9}, we have
  \begin{equation}\label{4.10}
    E_\al(v(0))+\varliminf_{n\to \infty}\int_{-t_n}^0e^{\al s} G_\al(v_n(s))\td s+\int_{-\infty}^0e^{\al s}\left(N_\al(v(s))+\Phi_\al(v(s))\right)\td s\le 0.
  \end{equation}

  On the other hand, passing to the limit as $n\to \infty$ in equation \eqref{eq1} for $v_n$, we can find that $v$ satisfies equality
  $$v_{tt}-\Delta v+f(v)+\tilde{g}=\phi.$$

  As stated in Remark \ref{rem_fu}, the boundedness of $\{v_n\}$ in $L_{loc}^{k+1}(\R;L^{3k+3}(\Omega))$ ensures that $\{f(v_n)\}$ is bounded in $L_{loc}^{(m+1)/m}(\R;L^{{(m+1)/m}}(\Omega))$, which yields that  $f(v)\in L_{loc}^{(m+1)/m}(\R;L^{{(m+1)/m}}(\Omega))$.
  Since $v_t\in L^{m+1}(\R\times\Omega)$, we can take $v_t+\al v$ as a test function in equation $v_{tt}-\Delta v=-f(v)-\tilde{g}+\phi$. Repeating the derivation of \eqref{eq4.12}, we get
  \begin{equation}\label{4.11}
    E_\al(v(0))+\int_{-\infty}^0e^{\al s}\int_\Omega \left(\tilde{g}v_t-\frac{3\al}2 v_t^2+\alpha \tilde{g}v\right)\td x\td s+\int_{-\infty}^0e^{\al s}\left(\Phi_\al(v)+N_\al(v)\right)\td s=0.
  \end{equation}
  Now, comparing with inequalities \eqref{4.10} and \eqref{4.11} gives
  \begin{equation}\label{4.12}
    \int_{-\infty}^0e^{\al s}\int_\Omega \left(\tilde{g}v_t-\frac{3\al}2 |v_t|^2+\alpha \tilde{g}v\right)\td x\td s \ge \varliminf_{n\to \infty}\int_{-t_n}^0e^{\al s} G_\al(v_n(s))\td s.
  \end{equation}
  Recalling that
  \begin{equation}\label{newdef:G}
    G_\al(v_n(s))=\int_\Omega \left(g(\p_t v_n)\p_t v_n-\frac{3\al}2 |\p_t v_n|^2+\alpha g(\p_t v_n)v_n\right)\td x.
  \end{equation}
  For the second term in $G_\al(v_n(s))$, by inequality $s^2\le c(\delta)g(s)s+\delta$, we have
  $$|\partial_t v_n|^2 \le c(\delta)g(\partial_t v_n)\partial_t v_n+\delta,$$
  which implies that
  \begin{equation}\label{4.14}
    \begin{aligned}
      \al \int_{-t_n}^0e^{\al s}\int_\Omega |\partial_t v_n|^2\td x\td s 
      \le \al c(\delta) \int_{-t_n}^0e^{\al s}\int_\Omega g(\partial_t v_n)\partial_t v_n \td x \td s+ \delta m (\Omega).
    \end{aligned}
  \end{equation}
  For the term involving $g(\partial_t v_n) v_n$ in $G_\al(v_n(s))$, H\"older inequality implies
  \begin{equation}
    \int_\Omega \left|g(\partial_t v_n) v_n\right|\td x
    \le \al \|v_n\|_{m+1}\|g(\partial_t v_n)\|_{{\frac{m+1}{m}}}.
  \end{equation}
  Then we utilize the inequality $|g(s)|^{(m+1)/m}\le c(\delta)g(s)s+\delta$ to obtain
  \begin{equation}\label{4.16}
    \begin{aligned}
      \|g(\partial_t v_n)\|_{\frac{m+1}{m}} & \le \(\int_\Omega \(c(\delta)g(\partial_t v_n)\partial_t v_n+\delta\) \td x\)^{\frac{m}{m+1}}                                \\
                                            & \le \delta^{-\frac{1}{m+1}} \int_\Omega (c(\delta)g(\partial_t v_n)\partial_t v_n+\delta) \td x +\delta^{\frac{m}{m+1}}      \\
                                            & \le \delta^{-\frac{1}{m+1}} c(\delta)\int_\Omega g(\partial_t v_n)\partial_t v_n \td x +\delta^{\frac{m}{m+1}}(m(\Omega)+1).
    \end{aligned}
  \end{equation}
  Thus
  \begin{equation}\label{4.17}
    \begin{aligned}
       & \int_{-t_n}^0e^{\al s}\int_\Omega \alpha \left|g(\partial_t v_n) v_n\right|\td x\td s                                                                                                                    \\
       & \le \int_{-t_n}^0e^{\al s}\al \|v_n\|_{m+1}\|g(\partial_t v_n)\|_{{\frac{m+1}{m}}} \td s                                                                                                                 \\
       & \le \al\delta^{-\frac{1}{m+1}} c(\delta)\|\mathscr{B}_0\|_{\H} \int_{-t_n}^0e^{\al s} \int_\Omega g(\partial_t v_n)\partial_t v_n \td x\td s +\delta^{\frac{m}{m+1}}\|\mathscr{B}_0\|_{\H}(m(\Omega)+1).
    \end{aligned}
  \end{equation}
  Collecting \eqref{newdef:G}, \eqref{4.14} and \eqref{4.17}, we conclude that
  \begin{equation}\label{4.18}
    \varliminf_{n\to\infty}\int_{-t_n}^0e^{\al s}G_\al(v_n)\td s \ge (1-\al a(\delta))\varliminf_{n\to\infty}\int_{-t_n}^0e^{\al s}\int_\Omega g(\partial_t v_n)\partial_t v_n \td x\td s-b(\delta),
  \end{equation}
  where
  $$a(\delta)=\frac32 c(\delta)+\delta^{-\frac{1}{m+1}}c(\delta)\|\mathscr{B}_0\|_{\H}$$
  and
  $$b(\delta)=\frac32\delta m(\Omega)+\delta^{\frac{m}{m+1}}\|\mathscr{B}_0\|_{\H}(m(\Omega)+1).$$
  According to \eqref{4.16},
  \begin{equation}\label{4.19}
    \begin{aligned}
          & \int_{-\infty}^0e^{\al s}\int_\Omega \alpha \tilde{g}v\td x\td s                                                                                                                                                         \\
      =   & \lim_{n\to \infty}\int_{-t_n}^0e^{\al s}\int_\Omega \alpha g(\partial_t v_n) v\td x\td s                                                                                                                                 \\
      \le & \al\delta^{-\frac{1}{m+1}} c(\delta)\|\mathscr{B}_0\|_{\H} \varliminf_{n\to\infty}\int_{-t_n}^0e^{\al s}\int_\Omega g(\partial_t v_n)\partial_t v_n \td x\td s+\delta^{\frac{m}{m+1}}\|\mathscr{B}_0\|_{\H}(m(\Omega)+1) \\
      \le & \al a(\delta)\varliminf_{n\to\infty}\int_{-t_n}^0e^{\al s}\int_\Omega g(\partial_t v_n)\partial_t v_n \td x\td s+b(\delta).
    \end{aligned}
  \end{equation}
  Substituting \eqref{4.18} and \eqref{4.19} into \eqref{4.12}, we have
  \begin{equation}\label{4.20}
    \int_{-\infty}^0e^{\al s}\int_\Omega \tilde{g}v_t\td x\td s
    \ge (1-2\al a(\delta))\varliminf_{n\to\infty}\int_{-t_n}^0e^{\al s}\int_\Omega g(\partial_t v_n)\partial_t v_n \td x\td s-2b(\delta).
  \end{equation}
  By the energy estimate \eqref{weak_energy_est}, the following dissipative relation
  \begin{equation}
    \int_{-t_n}^0\int_\Omega g(\partial_t v_n)\partial_t v_n \td x \td s\le Q(\|\mr{B}_0\|_{\H}),\forall n\in \N
  \end{equation}
  holds, which yields $\int_{-t_n}^0e^{\al s}\int_\Omega g(\partial_t v_n)\partial_t v_n \td x\td s$ is uniformly bounded for any $\al \ge 0$.
  Subsequently, we take the limit $\alpha \to 0$ in \eqref{4.20} to derive
  \begin{equation}
    \int_{-\infty}^0\int_\Omega \tilde{g}v_t\td x\td s
    \ge \varliminf_{n\to\infty}\int_{-t_n}^0\int_\Omega g(\partial_t v_n)\partial_t v_n \td x\td s-2b(\delta).
  \end{equation}
  Then passing the limit $\delta\to 0$, we obtain
  \begin{equation}\label{4.23}
    \int_{-\infty}^0\int_\Omega \tilde{g}v_t\td x\td s\ge \varliminf_{n\to \infty}\int_{-t_n}^0\int_\Omega g(\p_t v_n)\p_t v_n\td x\td s.
  \end{equation}
  Base on this, the monotonicity argument of Lions\cite{Lions65} implies $\tilde{g}=g(v_t)$, i.e.
  \begin{equation}\label{4.24}
    g(\p_t v_n)\wto g(v_t) \text{ weakly in } L^{\frac{m+1}{m}}(\R\times\Omega).
  \end{equation}
  Following this, we conclude that $(v(t),\p_t v(t))=y(t)$ is a complete trajectory and the equality \eqref{4.11} will turn into
  \begin{equation}\label{4.15}
    E_\al(v(0))+\int_{-\infty}^0e^{\al s}\left(G_\al(v)+\Phi_\al(v)+N_\al(v)\right)\td s=0.
  \end{equation}


  Proceeding, we will prove the compactness of $\{v_n(0)\}$. For the first term in \eqref{newdef:G}, by \eqref{4.24}, we have
  \begin{equation}
    \begin{aligned}
        & \varliminf_{n\to\infty}\int_{-t_n}^0e^{\al s} \int_\Omega \left(g(\partial_t v_n)\partial_t v_n- g(v_t )v_t\right) \td x\td s \\
      = & \varliminf_{n\to\infty}\int_{-t_n}^0e^{\al s}\int_\Omega(g(\partial_t v_n)-g(v_t ))(\partial_t v_n-v_t)\td x\td s.
    \end{aligned}
  \end{equation}
  Since $g$ is increasing, then
  \begin{equation}\label{4.26}
    \varliminf_{n\to\infty}\int_{-t_n}^0e^{\al s} \int_\Omega g(\partial_t v_n)\partial_t v_n\td x\td s\ge \int_{-\infty}^0e^{\al s}\int_\Omega g(v_t )v_t \td x\td s.
  \end{equation}
  Now, let $\al$ be small enough such that $1-\al a(\delta)>0$. It then follows from \eqref{4.18} that
  \begin{equation}\label{eq4.19}
    \begin{aligned}
      \varliminf_{n\to\infty}\int_{-t_n}^0e^{\al s}G_\al(v_n)\td s & \ge (1-\al a(\delta))\varliminf_{n\to\infty}\int_{-t_n}^0e^{\al s}\int_\Omega g(\partial_t v_n)\partial_t v_n \td x\td s-b(\delta) \\
                                                                   & \ge (1-\al a(\delta))\int_{-\infty}^0e^{\al s}\int_\Omega g(v_t) v_t \td x\td s-b(\delta).
    \end{aligned}
  \end{equation}
  By \eqref{4.19}, we obtain
  \begin{equation}\label{eq4.20}
    \begin{aligned}
      \int_{-\infty}^0e^{\al s}\left(\int_\Omega  g(v_t) v_t\td x-G_\al(v)\right) \td s & \ge -\int_{-\infty}^0e^{\al s}\int_\Omega \alpha g(v_t)v\td x\td s                        \\
                                                                                        & \ge -\al a(\delta) \int_{-\infty}^0e^{\al s}\int_\Omega  g(v_t) v_t \td x\td s-b(\delta).
    \end{aligned}
  \end{equation}
  Substituting \eqref{eq4.20} into \eqref{eq4.19}, we get:
  \begin{equation}\label{eq4.22}
    \varliminf_{n\to\infty}\int_{-t_n}^0e^{\al s}G_\al(v_n)\td s \ge \int_{-\infty}^0e^{\al s} G_\al(v) \td s -2\al a(\delta)\int_{-\infty}^0e^{\al s}\int_\Omega  g(v_t) v_t \td x\td s-2b(\delta).
  \end{equation}
  Combining with inequalities \eqref{eq4.13} and \eqref{eq4.22} gives
  \begin{equation}\label{eq4.23}
    \begin{aligned}
          & \varliminf_{n\to \infty}  \int_{-t_n}^0e^{\al s}\left(G_\al(v_n(s))\td s+N_\al(v_n(s)) \td s + \Phi_\al(v_n(s))\right)\td s                                 \\
      \ge & \int_{-\infty}^0e^{\al s}\left(G_\al(v)+\Phi_\al(v)+N_\al(v)\right)\td s-2\al a(\delta)\int_{-\infty}^0e^{\al s}\int_\Omega g(v_t)v_t\td x\td s-2b(\delta).
    \end{aligned}
  \end{equation}
  Thus passing the lower limit as $n\to \infty$ in the equality \eqref{eq4.12}  yields
  \begin{equation}
    \begin{aligned}
      0  = & \varliminf_{n\to \infty} \left(E_\al(v_n(0))+\int_{-t_n}^0e^{\al s}\left(G_\al(v_n(s)) + N_\al(v_n(s))+\Phi_\al(v_n(s))\right)\td s\right) \\
      \ge  & \varliminf_{n\to \infty}  E_\al(v_n(0))+\int_{-\infty}^0e^{\al s}\left(G_\al(v)+\Phi_\al(v)+N_\al(v)\right)\td s                           \\
           & -2\al a(\delta)\int_{-\infty}^0e^{\al s}\int_\Omega g(v_t)v_t\td x\td s-2b(\delta).
    \end{aligned}
  \end{equation}
  Recalling that from \eqref{4.19}, we have
  \begin{equation}
    \int_{-\infty}^0e^{\al s}\left(G_\al(v)+\Phi_\al(v)+N_\al(v)\right)\td s=-E_\al(v(0)).
  \end{equation}
  In conclusion, the following inequality holds:
  \begin{equation*}
    E_\al(v(0))\le \varliminf_{n\to \infty}E_\al(v_n(0)) \le E_\al(v(0))+2\al a(\delta)\int_{-\infty}^0e^{\al s}\int_\Omega g(v_t)v_t\td x\td s+2b(\delta).
  \end{equation*}
  Note that $\int_{-\infty}^0 \int_\Omega g(v_t)v_t\td x\td s$ is finite, so we can take the limit $\alpha \to 0$ to derive
  \begin{equation} \label{new:23}
    E(v(0))\le \varliminf_{n\to \infty}E(v_n(0)) \le E(v(0))+2b(\delta).
  \end{equation}
  Since \eqref{new:23} holds for any $\delta>0$, passing the limit $\delta\to 0$, we obtain the convergence of energy
  \begin{equation*}
    \varliminf_{n\to \infty}E(v_n(0))=E(v(0)),
  \end{equation*}
  which implies that $$\varliminf_{n\to \infty}\|S(t_n)z_n\|_{\H}=\|y(0)\|_{\H}.$$ Recall that weak convergence $S(t_n)z_n\wto y(0)\text{~in~}\H$, therefore there exists a strong convergent subsequence $S(t_n)z_n\to y(0)$ strongly in $\H$. Thus the asymptotic compactness of the semigroup $S(t)$ is verified. The proof is complete.
\end{proof}

Finally,   our main result in this section on the existence of global attractors follows directly from Lemma \ref{le4.1} and Lemma \ref{le4.6}, and classical attractor existence theory.
\begin{thm}\label{th4.4}
  Let the assumptions of Theorem \ref{weak} hold. Then, the semigroup $\{S(t)\}_{t \ge 0}$ generated by the weak solution of problem \eqref{eq1} possesses a global attractor $\mr{A}$ in the phase space $\H$.
\end{thm}

\section{Regularity of the global attractor}
In this section, we investigate the regularity of the global attractor $\mr{A}$ for the semigroup $\{S(t)\}_{t \ge 0}$ generated by the weak solution of problem \eqref{eq1}. To that end, the Assumption ({\bf G})  needs to be updated to include:
\begin{equation}\label{G2}
  0 <   \gamma\le g^\prime(s)\le c(1+g(s)s)^{2/3}.\tag{G2}
\end{equation}
It can be observed that
\begin{equation}\label{gs_gamma}
  \gamma s^{2}\le g(s)s~~\text{and  } |s|^{m+1}\le  C g(s)s \,\, \mbox{for some} \,\, C > 0.
\end{equation}

The regularity of the global attractor will be demonstrated in three steps. First, a uniform estimate of the strong solution is established in Lemma \ref{unif_strogn_est} below. Next,  partial regularity of bounded complete trajectories is proved in Lemma \ref{lem:pr} below. This is  followed by regularity estimates in $\V$ of the global attractor, which are shown in Lemma \ref{lem5.5} and Lemma \ref{lem5.6}. Finally,  the main result of this section is presented in Theorem \ref{regularity}.

\begin{lemma}[Uniform energy estimate of the strong solution]\label{unif_strogn_est}
  Assume that $\phi\in L^2(\Omega)$ and updated Assumption {\rm({\bf G})} with conditions \eqref{G1}-\eqref{G2} and Assumption {\rm({\bf F})} with conditions \eqref{F1}-\eqref{F3} hold. Then for any bounded subset $B$ of $\V$, the strong solutions are uniformly bounded in $\V$, i.e.,
  $$\sup_{t\ge0}\sup_{\varphi \in B}\|S(t)\varphi\|_{\V}< \infty.$$
\end{lemma}
\begin{proof}
  To derive the uniform bound, multiply equation \eqref{eq1} by $-\Delta u_t-\al \Delta u$. However, since $-\Delta u_t$ belongs to $H^{-1}$,  it cannot directly serve as a test function. Therefore, a formal derivation of the a priori estimate is presented next, which is justified using the Galerkin approximation method described in Theorem \ref{strong}.

  Multiplying equation \eqref{eq1} by $-\Delta u_t-\alpha\Delta u$ and integrating over $\Omega$, the following equality is obtained:

  \begin{equation}
    \begin{aligned}
       & \frac{\td}{\td t}\left(\frac12\|\nabla u_t\|^2+\frac12\|\Delta u\|^2-\int_\Omega \phi\Delta u\td x+\al\int_\Omega \nabla u_t\nabla u\td x+\frac12\int_\Omega f^\prime(u)|\nabla u|^2\td x\right) \\
       & +\int_\Omega g^{\prime}(u_t)|\nabla u_t|^2 \td x+\al \|\Delta u\|^2-\al\|\nabla u_t\|^2 +\al\int_\Omega g^{\prime}(u_t)\nabla u_t\nabla u\td x                                                   \\
       & -\int_\Omega f^{\prime\prime}(u)|\nabla u|^2 u_t\td x+\al\int_\Omega f^\prime(u)|\nabla u|^2\td x+\al\int_\Omega \phi \Delta u\td x=0.
    \end{aligned}
  \end{equation}
  Since $f^{\prime}(s)\ge -\omega s^4-K_\omega$, we can rewrite above equality as follows
  \begin{equation}\label{du}
    \begin{aligned}
       & \frac{\td}{\td t}\Phi(u)+\al\Phi(u)+\frac\al2 \|\Delta u\|^2-\frac{3\al}{2}\|\nabla u_t\|^2 +\int_\Omega g^{\prime}(u_t)|\nabla u_t|^2 \td x+\al\int_\Omega g^{\prime}(u_t)\nabla u_t\nabla u\td x \\
       & -\al^2\int_\Omega \nabla u_t\nabla u\td x-\frac12\int_\Omega f^{\prime\prime}(u)|\nabla u|^2 u_t\td x+\frac\al2\int_\Omega f^\prime(u)|\nabla u|^2\td x=2\al\|\phi\|^2,
    \end{aligned}
  \end{equation}
  where
  \begin{equation*}
    \Phi(u)=\frac12\|\nabla u_t\|^2+\frac12\|\Delta u\|^2-\int_\Omega \phi\Delta u\td x+\al\int_\Omega \nabla u_t\nabla u\td x+\frac12\int_\Omega f^\prime(u)|\nabla u|^2\td x+\frac{K_\omega}2\|\nabla u\|^2+2\|\phi\|^2.
  \end{equation*}
  By choosing $\omega>0$ and $\al>0$ both sufficiently small,  and using the energy estimate \eqref{weak_energy_est}, it follows that
  \begin{equation*}
    \frac14\|\nabla u_t\|^2+ \frac14\|\Delta u\|^2 \le \Phi(u)\le C_1(\|\nabla u_t\|^2+ \|\Delta u\|^2+Q(\|B\|_\H+\|\phi\|)).
  \end{equation*}
  Each term in equality \eqref{du} is analyzed as follows.

  First, by the Cauchy inequality, it is obtained that
  \begin{equation}\label{g'u_tu}
    \begin{aligned}
      \left|\int_\Omega g^{\prime}(u_t)\nabla u_t\nabla u\td  x\right| & \le \frac12\int_\Omega g^\prime(u_t)|\nabla u_t|^2\td x+\frac{1}2\int_\Omega g^\prime(u_t)|\nabla u|^2\td x.
    \end{aligned}
  \end{equation}
  Condition \eqref{G2} and H\"older inequality imply that
  \begin{equation}\label{g'du2}
    \begin{aligned}
      \int_\Omega g^\prime(u_t)|\nabla u|^2\td x & \le C_2\int_\Omega \left(1+(g(u_t)u_t)^{2/3}\right)|\nabla u|^2\td x                     \\
                                                 & \le C_2\|\nabla u\|^2+C_2\left(\int_\Omega g(u_t)u_t\td x\right)^{2/3} \|\nabla u\|_6^2.
    \end{aligned}
  \end{equation}
  By Jensen's inequality and Sobolev's embedding
  \begin{equation}\label{5.6}
    \begin{aligned}
        C_2\left(\int_\Omega g(u_t)u_t\td x\right)^{2/3} \|\nabla u\|_6^2 &\le \left(\frac12+C\int_\Omega g(u_t)u_t\td x\right)\|\Delta u\|^2\\
        &\le \frac12\|\Delta u\|^2+C\int_\Omega g(u_t)u_t\td x\Phi(u).
    \end{aligned}
  \end{equation}
  Recall that $g^{\prime}(s)\ge \gamma$, therefore
  \begin{equation}\label{g'}
    \begin{aligned}
          & \int_\Omega g^{\prime}(u_t)|\nabla u_t|^2 \td x+\al\int_\Omega g^{\prime}(u_t)\nabla u_t\nabla u\td x              \\
      \ge & \gamma(1-\frac\al2)\|\nabla u_t\|^2-\frac\al4\|\Delta u\|^2-C\al\int_\Omega g(u_t)u_t\td x\Phi(u)-C\|\nabla u\|^2.
    \end{aligned}
  \end{equation}
  Similar to \eqref{fpp_du_ut}, by Lemma \ref{inter_ineq},   for any $\theta>0$, there exists a constant $C_\theta>0$ such that
  \begin{equation}\label{eq5.6}
    \begin{aligned}
      \int_\Omega f^{\prime\prime}(u)|\nabla u|^2 u_t\td x & \le \left(\theta+\theta\|u\|^{k+1}_{3k+3}+C_{\theta}\left(1+\|u\|^{(p-2)(m+1)}_{6}\right)\|u_t\|^{m+1}_{m+1}\right)\|\Delta u\|^2 \\
                                                           & \le 4\left(\theta+\theta\|u\|^{k+1}_{3k+3}+C_{\theta}\left(1+\|u\|^{(p-2)(m+1)}_{6}\right)\|u_t\|^{m+1}_{m+1}\right)\Phi(u)       \\
                                                           & :=H(t)\Phi(u).
    \end{aligned}
  \end{equation}

  Substituting inequalities \eqref{g'}-\eqref{eq5.6} into equality \eqref{du}, the following inequality is obtained:
  \begin{equation}\label{}
    \frac{\td}{\td t}\Phi(u)+\al\Phi(u)\le \left(C\al\int_\Omega g(u_t)u_t\td x+\frac12H(t)\right) \Phi(u)+Q(\|B\|_{\H}+\|\phi\|).
  \end{equation}
  From the energy estimate \eqref{weak_energy_est} it follows that
  \begin{equation}
    \int_{t_1}^{t_2}H(t)\td t\le \theta (t_2-t_1)Q(\|B\|_{\H}+\|\phi\|)+Q(\|B\|_{\H}+\|\phi\|), ~~ \forall \,\,  t_1, t_2 \in [0, \infty).
  \end{equation}
  Now choose $\theta=\al/Q(\|B\|_\H+\|\phi\|)$. Since $\int_{0}^{+\infty} \int_\Omega g(u_t)u_t\td x\td t\le Q(\|B\|_\H+\|\phi\|)$,  it then follows that
  \begin{equation}
    \int_{t_1}^{t_2}\left(C\al\int_\Omega g(u_t)u_t\td x+H(t)\right)\td t\le \al(t_2-t_1)+Q(\|B\|_\H+\|\phi\|).
  \end{equation}
  Finally, applying Gronwall's Lemma to inequality \eqref{eq5.6} implies that
  \begin{equation}\label{Phi}
    \Phi(u(t))\le Q(\|B\|_\H+\|\phi\|)|\Phi(u(0))|e^{-\al t/2}+Q(\|B\|_\H+\|\phi\|),~~ \text{for any } t\ge 0,
  \end{equation}
  which completes the proof.
\end{proof}

\begin{lemma}[Partial regularity of bounded complete trajectories]\label{lem:pr}
  Under the same assumptions as in Lemma \ref{unif_strogn_est}, for any trajectory $U(t)=(u(t), u_t(t))$ in $\mathscr{A}$, we have
  \begin{equation*}\label{part_reg}
    \|u_{tt}(t)\|+\|\nabla u_t(t)\| \le Q(\|\mathscr{A}\|_{\H})~~\text{for all }t\in\R,
  \end{equation*}
  which implies that the global attractor $\mathscr{A}$ is bounded in $H^1_0(\Omega)\times H^1_0(\Omega)$.
\end{lemma}
\begin{proof}
  To establish the partial regularity,   the concept of difference quotients is utilized instead of directly differentiating equation \eqref{eq1}. Define the difference quotient as follows:
  $$D_h u(t)=\dfrac{u(t+h)-u(t)}{h} ~~ \text{for } h > 0.$$
  The function $D_h u(t)$ satisfies the following differential equation:
  \begin{equation}\label{dhu}
    \begin{cases}
       & (D_h u)_{tt} -\Delta D_h u +f_1 (u(t),h)D_h u +g_1 (u_t(t),h)(D_h u)_t = 0 ~~ \text{in } \Omega\times(t_0,\infty), \\
       & D_h u = 0 ~~ \text{on } \partial \Omega \times (t_0,\infty) ~~\text{for } t_0\in \R,
    \end{cases}
  \end{equation}
  where
  \begin{eqnarray*} f_1(u(t),h) &=& \int_0^1 f^\prime(su(t+h)+(1-s)u(t))\td s , \\
    g_1 (u_t ,h) &=& \int_0^1 g^\prime(su_t(t+h) + (1-s)u_t(t))\td s.
  \end{eqnarray*}

  By multiplying both sides of \eqref{dhu} by $D_h u_t+ \al D_h u$ and integrating over $\Omega$, a similar proof to that of \eqref{Phi} on $[t_0, t]$ is applied to obtain
  \begin{equation*}
    \begin{aligned}
      \|D_h u_{t}(t)\|^2+\|\nabla D_h u(t)\|^2
      \le  Q(\|\mathscr{A}\|_{\H})\left(\|D_h u_{t}(t_0)\|^2+\|\nabla D_h u(t_0)\|^2\right)e^{\frac{\al (t_0-t)}{8}}+Q(\|\mathscr{A}\|_{\H}),
    \end{aligned}
  \end{equation*}
  for any $t\ge t_0$ and for some $\al>0$. Since for any fixed $h>0$,
  $$\|D_h u_{t}(t_0)\|^2+\|\nabla D_h u(t_0)\|^2\le Q(\|\mathscr{A}\|_{\H})/h^2 ~~ \text{uniformly for any  } t_0\in \R,$$
  letting $t_0\to -\infty$ gives
  \begin{equation*}
    \|D_h u_{t}(t)\|^2+\|\nabla D_h u(t)\|^2\le Q(\|\mathscr{A}\|_{\H}) ~~\text{for any } h>0 \text{~and } t\in \R.
  \end{equation*}
  Consequently, a standard convergence argument gives
  \begin{equation*}
    \|u_{tt}(t)\|+\|\nabla u_t(t)\| \le Q(\|\mathscr{A}\|_{\H})~~\text{for all }t\in\R.
  \end{equation*}
  This demonstrates that the trajectory $U(t) = (u(t), u_t(t))$ in the global attractor $\mathscr{A}$  maintains partial regularity, and completes the proof.
\end{proof}

To establish the regularity estimate in $\V$ of the global attractor, the cutting-off method of Khanmamedov \cite{khanmamedov10} is employed and refined. More precisely,
Since $\mathscr{A} = S(t)\mathscr{A}$ for all $t > 0$, for any point $(u_0, u_1)\in \mathscr{A}$, there exists a trajectory $z(t)$ passing through this point  such that $z(t)=(u(t), u_t(t))\in \mathscr{A}$ for all $t\in R$. Given that $\mathscr{A}$ is compact, there exists a decreasing negative sequence $\{t_m\}$ with $t_m\to -\infty$ and a point $z_0=(\bar{u}_0, \bar{u}_1)\in \mathscr{A}$, such that $z(t_m)\to z_0$ in $\H$ as $m\to \infty$. Consequently,  $z(t+t_{m})=S(t)z(t_m)\to S(t)z_0$ for any $t\ge 0$.

Moreover, the dynamical system $(S(t), \H)$ is a gradient system with a strict Lyapunov functional given by
$$\mathcal{L}(z(t))=\frac12\|u_t\|^2+\frac12\|\nabla u\|^2+\int_\Omega F(u)\td x-\int_\Omega \phi u\td x.$$
It  then follows  that
$$\mathcal{L}(z_0)=\lim_{m\to\infty}\mathcal{L}(z(t_{m}))=\lim_{m\to\infty}\mathcal{L}(z(t+t_{m}))=\mathcal{L}(S(t)z_0)~~ \text{for any  } t\ge 0,$$
which implies that $z_0$ is an equilibrium point of semigroup $S(t)$. Thus $\bar{u}_0$ satisfies the elliptic equation $-\Delta \bar{u}_0+f(\bar{u}_0)=\phi$ and $\bar{u}_1=0$. It is evident that the set of equilibrium points of semigroup $S(t)$
is bounded in $\V$.

Next, introduce the decomposition $u(t)=w_n(t)+v_n(t)+\bar{u}_0(x)$, where

\begin{equation}\label{wn}
  \begin{cases}
    \partial_t^2 w_n+g( \partial_{t} w_n)-\Delta w_n+f_n(u)-f(\bar{u}_0)=0,~~\text{in }(t_m,\infty)\times\Omega, \\
    w_n=0,~~\text{on }(t_m,\infty)\times \partial\Omega,                                                         \\
    w_n(t_m,x)=0,~~ \partial_{t} w_n(t_m,x)=0,~~\text{in }\Omega,
  \end{cases}
\end{equation}
and
\begin{equation}\label{vn}
  \begin{cases}
    \partial_t^2 v_n+g(u_t)-g(\partial_{t} w_n)-\Delta v_n+f(u)-f_n(u)=0,~~\text{in }(t_m,\infty)\times\Omega, \\
    v_n=0,~~\text{on }(t_m,\infty)\times\partial\Omega,                                                        \\
    v_n(t_m,x)=u(t_m,x)-\bar{u}_0(x),~~ \partial_{t} v_n(t_m,x)=u_t(t_m,x),~~\text{in }\Omega.
  \end{cases}
\end{equation}
Here, $f_n(s)$ is defined as
$$f_n(s)=\begin{cases}
    f(n),~~ s\geq n, \\
    f(s),~~ |s|<n,   \\
    f(-n),~~ s\leq-n,
  \end{cases}
  n\in\mathbb{N}.$$

This decomposition and the corresponding boundary conditions facilitate the derivation of regularity estimates for the components $w_n$ and $v_n$, thereby establishing the desired regularity estimate in $\V$ for the global attractor in the sequel.

\begin{lemma}\label{lem5.5}
  Under the same assumptions as in Lemma \ref{unif_strogn_est},  we have $(w_{n}(t), \partial_{t} w_n(t)) \in \V$. Moreover,  for any $n\in \mathbb{N}$ there exists a time
  $T_n < 0$ such that
  \begin{equation}\label{wH2}
    \|\nabla \partial_{t} w_n(t)\|+\|\Delta w_{n}(t)\| \leq r_{0}n^2\sup_{t_m\le s\le T_n}\|\nabla v_n(s)\|+r_0,~~\forall t_m\leq t\leq T_{n},
  \end{equation}
  where the positive constant $r_0$ is independent of $m$, $n$ and $(u_0, u_1)$.
\end{lemma}

\begin{proof}
  The proof consists of three parts: in the first part an $H^1_0(\Omega)\times H^1_0(\Omega)$ estimate for $w_n(t)$ is derived,  in the second part a dissipativity relation is established, and in the last part an $H^2(\Omega)$ estimate of $w_n(t)$ is constructed.\\
  (i) {\it $H^1_0(\Omega)\times H^1_0(\Omega)$ estimate.}  \quad   Differentiating equation \eqref{wn} with respect to $t$ and multiplying both sides of the derived equation by $\partial_t^2 w_n$, integrating over $(t_m, T)\times \Omega$ yields
  \begin{equation}\label{5.19}
    \begin{aligned}
          & \|\partial_t^2 w_n(t)\|^2+\|\nabla \partial_t w_{n}(t)\|^2                                                                           \\
      \le & C_1\left(\|\partial_t^2 w_n(t_m)\|^2+\|\nabla \partial_t w_{n}(t_m)\|^2+\int_{t_m}^{t}\|f_n^{\prime}(u(t))u_t(t)\|^{2}\td t+1\right) \\
      \le & C_1\left(\|\partial_t^2 w_n(t_m)\|^2+n^{8}\int_{t_m}^{t}\|u_{t}(t)\|^{2}\td t+1\right),~~ \forall t\ge t_m.
    \end{aligned}
  \end{equation}
  From the equation \eqref{wn} for $w_n$,  it follows that  $\partial_t^2 w_n(t_m)=-g(0)+f(\bar{u}_0)-f_n(u(t_m))=f(\bar{u}_0)-f_n(u(t_m))$, and hence
  \begin{equation}
    \begin{aligned}
      \|\partial_t^2 w_n(t_m)\| & \le \|f_n(\bar{u}_0)-f_n(u(t_m))\|+\|f(\bar{u}_0)-f_n(\bar{u}_0)\| \\
                                & \le C_1n^2\|\bar{u}_0-u(t_m)\|_6+C_2.
    \end{aligned}
  \end{equation}
  Because $\|\bar{u}_0-u(t_m)\|_6\to 0$ as $m\to \infty$,   for any $n\in \mathbb{N}$ there exists $M_n\in \mathbb{N}$ such that $n^2\|\bar{u}_0-u(t_m)\|_6\le 1$ for all $m>M_n$, and therefore
  \begin{equation}\label{5.22}
    \|\partial_t^2 w_n(t_m)\|\le C_1+C_2, ~~ \forall m>M_n.
  \end{equation}
  As $u(t)$ is a global-in-time solution,  the dissipativity relation follows from the energy estimate \eqref{weak_energy_est}:
  $$\int_{-\infty}^{+\infty}\int_{\Omega}g(u_{t})u_{t}\td x\td t\le Q(\|\mathscr{A}\|_{\H}),$$
  which  along with \eqref{gs_gamma} leads to
  \begin{equation}\label{5.21}
    \int^{+\infty}_{-\infty}\|u_{t}(s)\|^{2}\td s\le \frac{1}{\gamma}\int_{-\infty}^{+\infty}\int_{\Omega}g(u_{t})u_{t}\td x\td t\le Q(\|\mathscr{A}\|_{\H}).
  \end{equation}
  Consequently, for any $n \in \mathbb{N}$ there exists $T_n<t_{M_n}<0$ such that
  \begin{equation}\label{5.24}
    n^{8}\int_{-\infty}^{T_n}\|u_{t}(s)\|^{2}\td s\le 1.
  \end{equation}
  Substituting \eqref{5.22} and \eqref{5.24} into \eqref{5.19},  it follows that
  \begin{equation*}\label{}
    \|\partial_t^2 w_n(s)\|+\|\nabla \partial_t w_{n}(s)\| \leq r_{1},~~\text{for any  } t_m\leq s\leq T_n,
  \end{equation*}
  which together with the equation \eqref{wn} yields
  \begin{equation}\label{5.26}
    \|\nabla \partial_{t} w_n(t)\|+\|\nabla w_{n}(t)\| \leq r_2,~~\text{for any  } t_m\leq t\leq T_n,
  \end{equation}
  where the positive constants $r_1$ and $r_2$ are independent of $m$, $n$ and $(u_0, u_1)$.\\
  (ii) {\it Dissipativity relation.} \quad  By multiplying both sides of \eqref{wn} by $\partial_{t} w_n$ and
  integrating over $(t_1 ,t_2) \times \Omega$, it yields
  \begin{equation}\label{new:30}
    \int_{t_1}^{t_2}\int_\Omega g(\partial_{t} w_n(t))\partial_{t} w_n(t)\td x\td t \leq C_{2}+\int_{t_1}^{t_2}\int_\Omega f_{n}^{\prime}(u(t))u_{t}(t) w_{n}(t) \td x\td t.
  \end{equation}
  Taking into accounts the estimates \eqref{5.26} and \eqref{5.24} it follows from \eqref{new:30} that
  \begin{equation}\label{diss}
    \begin{aligned}
      \int_{t_1}^{t_{2}}\int_\Omega g(\partial_{t} w_n(t))\partial_{t} w_n(t)\td x\td t & \le C_{2}+\|f_n^{\prime}(u(t))\|_\infty \|u_t(t)\|\cdot\|w_{n}(t)\|                               \\
                                                                                        & \le C_{2}+C_{3}n^4\int_{t_1}^{t_2}\|u_{t}(t)\|\td t                                               \\
                                                                                        & \le  C_{2}+C_{3}n^4\left(\int_{t_1}^{t_2}\|u_{t}(t)\|^2\td t\right)^{\frac12} (t_2-t_1)^{\frac12} \\
                                                                                        & \le  C_{4}(1+(t_2-t_1)^{\frac12}),~~ t_m\leq t_1\leq t_2\leq T_{n}.
    \end{aligned}
  \end{equation}
  (iii) {\it $H^2(\Omega)$ estimate of $w_{n}(t)$.} \quad  Multiplying both sides of equation \eqref{wn} by $- \Delta \partial_{t} w_n-\al\Delta w_n$, where $\al \in (0, 1)$, and integrating over $\Omega$ gives
  \begin{equation}\label{Delta_w}
    \begin{aligned}
       & \frac{\td}{\td t}\Phi_{n}(t) +\al\Phi_{n}(t)+\int_\Omega g^\prime(\partial_{t} w_n)|\nabla \partial_{t} w_n|^2\td x+\frac\al2\|\Delta w_{n}\|^2-\frac{3\al}{2}\|\nabla \partial_{t} w_{n}\|^{2}  \\
       & -\al^2\int_\Omega\nabla\partial_{t} w_{n}\nabla w_{n}\td x+\al\int_\Omega g^\prime (\partial_{t} w_n)\nabla w_{n}\nabla \partial_{t} w_n\td x-\int_\Omega f^\prime_{n}(u)u_t\Delta w_{n}\td x=0.
    \end{aligned}
  \end{equation}
  in which
  \begin{equation*}
    \Phi_{n}(t) : ={\frac{1}{2}}\|\nabla \partial_t w_{n}(t)\|^{2}+{\frac{1}{2}}\|\Delta w_{n}\|^{2}+\al\int_\Omega\nabla \partial_{t}w_{n}\nabla w_{n}\td x+\int_\Omega f^\prime_{n}(u)\nabla u\nabla w_{n}\td x+\la f(\bar{u}_0),\Delta w_{n}\ra .
  \end{equation*}

  Regarding the fourth term of $\Phi_n(t)$, one has
  \begin{equation}
    \begin{aligned}
      \int_\Omega f^\prime_{n}(u)\nabla u\nabla w_{n}\td x & =\int_\Omega f^\prime_{n}(u)\nabla v_n\nabla w_{n}\td x+\int_\Omega f^\prime_{n}(u)|\nabla w_n|^2\td x.
    \end{aligned}
  \end{equation}
  Notice that $f_n^\prime(s)=f^\prime(s)$ on $[-n, n]$ and $=0$ otherwise, it follows that $f_n^\prime(s)\ge -\omega s^4-K_\omega$. Therefore taking  into account \eqref{5.26},  one obtains
  \begin{equation}
    \int_\Omega f^\prime_{n}(u)\nabla u\nabla w_{n}\td x \ge -\|f^\prime_n(u)\|_3\|\nabla v_n\|\cdot\|\nabla w_n\|_6-\omega\|u\|_6^4\|\nabla w_n\|_6^2-K_\omega\|\nabla w_n\|^2.
  \end{equation}
  Now choosing $\omega>0$ small enough, one gets
  \begin{equation}\label{f2}
    \int_\Omega f^\prime_{n}(u)\nabla u\nabla w_{n}\td x \ge -C_5n^{4}\|\nabla v_n\|^2-\frac18\|\Delta w_n\|^2-C_6,
  \end{equation}
  which, along with \eqref{5.26} implies that
  \begin{equation}\label{5.34}
    \Phi_{n}(t)\ge \frac14\left(\|\nabla \partial_t w_{n}(t)\|^{2}+\|\Delta w_{n}(t)\|^{2}\right)-C_5n^{4}\|\nabla v_n\|^2-C_7,~~ t_m\leq t\leq T_{n},
  \end{equation}
  we just need the lower bound estimate of $\Phi_n(t)$ since $\Phi_n(t_m)=0$ .

  Regarding the first term on the last line of \eqref{Delta_w}, repeating the procedure from \eqref{g'u_tu} to \eqref{5.6} gives
  \begin{equation}\label{new:31}
    \begin{aligned}
      \int_\Omega g^\prime (\partial_{t} w_n)|\nabla w_{n}\nabla \partial_{t} w_n|\td x & \le \frac12\int_\Omega g^\prime (\partial_{t} w_n)|\nabla \partial_{t} w_n|^2\td x+C_6\|\nabla w_{n}\|^2 \\
                                                                                        & + \frac14\|\Delta w_{n}\|^2+C\int_\Omega g(\partial_{t} w_n)\partial_{t} w_n\td x \|\Delta w_{n}\|^2.
    \end{aligned}
  \end{equation}
  Using inequality \eqref{5.34} one has
  \begin{equation}\label{5.36}
    \|\Delta w_{n}(t)\|^{2} \le 4\Phi_{n}(t) +4C_5n^{4}\|\nabla v_n(t)\|^2+4C_7,~~ t_m\leq t \leq T_{n}.
  \end{equation}
  Inserting \eqref{5.36} in \eqref{new:31} yields
  \begin{equation}
    \begin{aligned}
      \int_\Omega g^\prime (\partial_{t} w_n)|\nabla w_{n}\nabla \partial_{t} w_n|\td x\le & \frac12\int_\Omega g^\prime (\partial_{t} w_n)|\nabla \partial_{t} w_n|^2\td x+C_6\|\nabla u\|^2+\frac14\|\Delta w_{n}\|^2 \\
      +                                                                                    & C_8\int_\Omega g(\partial_{t} w_n)\partial_{t} w_n\td x\left(\Phi_{n}(t)+C_5n^{4}\|\nabla v_n\|^2-C_7.\right).
    \end{aligned}
  \end{equation}
  For the second term on the last line of \eqref{Delta_w}, one has
  \begin{equation}\label{5.38}
    \int_\Omega f^\prime_{n}(u)u_t\Delta w_{n}\td x\le \frac\al4\|\Delta w_{n}\|^2+C_\al\|f^\prime_{n}(u)u_t\|^2\le \frac\al4\|\Delta w_{n}\|^2+C_\al n^8\|u_t\|^2.
  \end{equation}
  Choosing $\al$ to be sufficiently small in \eqref{5.38}, and combining with \eqref{Delta_w}, \eqref{f2} and $g^\prime(s)\ge \gamma$ results in
  \begin{equation*}\label{}
    \begin{aligned}
      \frac{\td}{\td t} \Phi_{n}(t)+ \al \Phi_{n}(t)\le & C_8\int_\Omega g(\partial_{t} w_n)\partial_{t} w_n\td x\left(\Phi_{n}(t)+C_5n^{4}\|\nabla v_n\|^2+C_7\right)
      +C_\al n^8\|u_t\|^2.
    \end{aligned}
  \end{equation*}
  Taking into account \eqref{5.24} and \eqref{diss}, it can be derived from Gronwall's Lemma that
  \begin{equation}\label{5.40}
    \Phi_{n}(t)\le C_9n^{4}\sup_{t_m\le s\le T_n}\|\nabla v_n\|^2+C_9,~~ t_m\leq t \leq T_{n}.
  \end{equation}
  The inequality \eqref{5.40} combined with \eqref{5.34} and \eqref{5.40} gives the desired inequality \eqref{wH2}, and completes the proof.
\end{proof}

\begin{lemma}\label{lem5.6}
  Under the same assumptions as in Lemma \ref{unif_strogn_est},  there exists $n_0 \in \mathbb{N}$ such that
  \begin{equation}\label{5.41}
    \lim_{m\to \infty}\sup_{t_m\le t\le T_{n_0}}\left(\|v_{n_0t}(t)\|+\|\nabla v_{n_0}(t)\|\right)=0.
  \end{equation}
\end{lemma}
\begin{proof}
  By multiplying both sides of equation \eqref{vn} by $\partial_{t} v_n + \al v_n$, where $\al \in (0, 1 ) $, and integrating over $\Omega$, the following equation is obtained:
  \begin{equation}\label{5.42}
    \begin{aligned}
       & {\frac{\td}{\td t}}\left(\frac12\|\partial_{t} v_n(t)\|^2+\frac12\|\nabla v_{n}(t)\|^2+\al \langle \partial_{t} v_n(t),v_{n}(t)\rangle\right)+\al \|\nabla v_{n}(t)\|^{2}-\al\| \partial_{t} v_n(t)\|^{2} \\
       & +\int_\Omega \left(g(u_t(t))-g(\partial_{t} w_n(t))\right)(\partial_{t} v_n(t)+\al v_{n}(t))\td x                                                                                                         \\
       & \leq  \|f(u(t))-f_n(u(t))\|_{6/5}\|\partial_{t} v_n(t)\|_{6}+\al \|f(u(t))-f_n(u(t))\|_{6/5}\|v_n(t)\|_6.
    \end{aligned}
  \end{equation}
  By Lemma \ref{lem:pr} and estimate \eqref{5.26}, one has uniform boundedness estimate
  $$\|\partial_{t} v_n(t)\|_{6}\le C,~~ t_m\leq t \leq T_{n}.$$
  Then for the last two terms in \eqref{5.42}, using Sobolev's embedding theorem gives
  \begin{equation*}
    \begin{aligned}
          & \|f(u(t))-f_n(u(t))\|_{6/5}\|\partial_{t} v_n(t)\|_{6}+\al \|f(u(t))-f_n(u(t))\|_{6/5}\|v_n(t)\|_6 \\
      \le & C\|f(u(t))-f_n(u(t))\|_{6/5}+ C\|f(u(t))-f_n(u(t))\|^2_{6/5}+\al^2\|\nabla v_n(t)\|^2              \\
      =   & \left(C+ C\|f(u(t))-f_n(u(t))\|_{6/5}\right)\|f(u(t))-f_n(u(t))\|_{6/5}+\al^2\|\nabla v_n(t)\|^2   \\
      \le & C_1\|f(u(t))-f_n(u(t))\|_{6/5}+\al^2\|\nabla v_n(t)\|^2.
    \end{aligned}
  \end{equation*}
  To estimate the fourth term on the left side of \eqref{5.42},  denote
  $$g_1(u_t,\partial_{t} w_n) =\int_0^1 g^\prime(su_t+(1-s)\partial_{t} w_n)\td s ~~\text{and~~} g_2(u_t,\partial_{t} w_n)=g(u_t)u_t+g(\partial_{t} w_n)\partial_{t} w_n.$$
  It then follows that
  $$0<\gamma \le g_1(u_t,\partial_{t} w_n) \le C[1+g_2(u_t,\partial_{t} w_n)]^{2/3},$$
  and consequently
  \begin{equation*}
    \begin{aligned}
           & \int_\Omega \left(g(u_t(t))-g(\partial_{t} w_n(t))\right)\partial_{t} v_n(t) \td x +\al\int_\Omega \left(g(u_t(t))-g(\partial_{t} w_n(t))\right)  v_{n}(t) \td x             \\
      =    & \int_\Omega g_1(u_t,\partial_{t} w_n)|\partial_{t} v_n(t)|^2 \td x + \al \int_\Omega g_1(u_t,\partial_{t} w_n) \partial_{t} v_n(t) v_{n}(t) \td x                            \\
      \ge  & \frac12\int_\Omega g_1(u_t,\partial_{t} w_n)|\partial_{t} v_n(t)|^2 \td x -\frac{\al^2}{2}\int_\Omega g_1(u_t,\partial_{t} w_n) |v_{n}(t)|^2 \td x                           \\
      \geq & \frac{\gamma}{2}\|\partial_{t} v_n(t)\|^2-\left(\frac{\al^2}{4}+\al^2C_2\int_\Omega g_2(u_t,\partial_{t} w_n)\td x\right) \|\nabla v_n(t)\|^2-\frac{1}{2}\al^2C\|v_n(t)\|^2.
    \end{aligned}
  \end{equation*}
  Thus,  taking this inequality  into account  in \eqref{5.42} and choosing $\al$ small enough yields
  \begin{equation}\label{5.44}
    \begin{aligned}
           & {\frac{\td}{\td t}}\left(\frac12\|\partial_{t} v_n(t)\|^2+\frac12\|\nabla v_{n}(t)\|^2+\al \langle \partial_{t} v_n(t),v_{n}(t)\rangle\right)+\frac\al4 \|\nabla v_{n}(t)\|^{2}+\frac\al4\| \partial_{t} v_n(t)\|^{2} \\
      \leq & C_2\al^2 \int_\Omega g_2(u_t,\partial_{t} w_n)\td x \|\nabla v_n(t)\|^2 +C_1\|f(u(t))-f_n(u(t))\|_{6/5}.
    \end{aligned}
  \end{equation}
  Since  $\int_{-\infty}^{+\infty}\int_{\Omega}g(u_{t})u_{t}\td x\td t\le Q(\|\mathscr{A}\|_{\H})$ and \eqref{diss} hold, it follows that
  \begin{equation*}
    \int_{t_1}^{t_{2}}\int_\Omega g_2(u_t,\partial_{t} w_n)\td x \td t \le  C_{3}(1+(t_2-t_1)^{\frac12}),~~ t_m\leq t_1\leq t_2\leq T_{n}.
  \end{equation*}
  Applying the Gronwall's Lemma to \eqref{5.44} gives
  \begin{equation*}
    \begin{aligned}
      \|\partial_{t} v_n(t)\|^2+\|\nabla v_{n}(t)\|^2 & \le C_4\left(\|\partial_{t} v_n(t_m)\|^2+\|\nabla v_{n}(t_m)\|^2\right)e^{-\al(t-t_m)/4} \\
                                                      & +C_4\sup_{t_m\le s\le T_n}\|f(u(s))-f_n(u(s))\|_{6/5}.
    \end{aligned}
  \end{equation*}

  Recalling that initial value $(v_n(t_m), \partial_{t} v_n(t_m))=(u(t_m)-\bar{u}_0, u_t(t_m)) \to (0, 0)$ in $\H$ as $m\to \infty$,  there exists $M^\prime_n\in \mathbb{N}$ such that $M^\prime_n>M_n$ and
  \begin{equation*}
    \|\partial_{t} v_n(t_m)\|^2+\|\nabla v_{n}(t_m)\|^2\le \sup_{t_m\le s\le T_n}\|f(u(s))-f_n(u(s))\|_{6/5}, ~~ \text{for all }m>M^\prime_n.
  \end{equation*}
  Therefore
  \begin{equation*}
    \|\partial_{t} v_n(t)\|^2+\|\nabla v_{n}(t)\|^2 \le 2C_4 \sup_{t_m\le s\le T_n}\|f(u(s))-f_n(u(s))\|_{6/5},
  \end{equation*}
  for any $m>M^\prime_n$ and $t\ge t_m$.

  Combining this with \eqref{wH2} and Agmon inequality $\|w_n\|_{\infty}^2\le c\|\nabla w_n\|\cdot\|\Delta w_n\|$, it follows that
  \begin{equation*}\label{}
    \|w_n\|_{\infty}\le  Cn\sup_{t_m\le s\le T_n}\|\nabla v_n(s)\|^{1/2}+C\le C_5 n\sup_{t_m\le s\le T_n}\|f(u(s))-f_n(u(s))\|^{1/4}_{6/5}+C_5.
  \end{equation*}
  Now denote $A_n(t)=\{x\in \Omega : |u(t,x)|\ge n\}$, we obtain
  \begin{equation}\label{5.51}
    \|f(u(t))-f_n(u(t))\|_{6/5} \leq \left(C_6\int_{A_{n}(t)} |u(t,x)|^{6} \td x\right)^{5/6}, ~~\forall t\in \R.
  \end{equation}
  Since $\mathscr{A}$ is a compact subset of $\H$ and $(u(t), u_t(t)) \in \mathscr{A}$, there exists $n_1 \in \mathbb{N}$ such that
  $$\sup_{t\in \R}\int_{A_{n_1}(t)} |u(t,x)|^6 \td x\leq \frac{1}{C_6\left({3C_5}\right)^{24/5}},$$
  and hence
  \begin{equation}\label{5.52}
    \|w_n(t)\|_{\infty}\le  \frac n3+C_5,~~\forall t_m\leq t\leq T_{n_1}.
  \end{equation}

  Let $\hat{A}(t)=\{x\in \Omega : |w_n(t)|\ge |v_n(t)|\}$, from \eqref{5.52} it is known that $A_n(t)\subset \hat{A}_n(t)$ for $n\ge n_1+6C_5$. Then refine the estimate \eqref{5.51} as follows
  \begin{equation}\label{5.49}
    \begin{aligned}
      \|f(u(t))-f_n(u(t))\|_{6/5} & \leq \left(C_6\int_{A_{n}(t)}|u(t)|^{6}\td x\right)^{\frac56}                                                         \\
                                  & \leq C\left(\int_{A_n(t)} |u(t)|^6 \td x\right)^{\frac12}\left(\int_{\hat{A}_n(t)}|u(t)|^6\td x\right)^{\frac13}      \\
                                  & \leq C\left(\int_{A_{n}(t)}|u(t)|^{6}\td x\right)^{\frac12} \left(\int_{\hat{A}_n(t)}|v_n(t)|^6\td x\right)^{\frac13} \\
                                  & \leq C_7\left(\int_{A_{n}(t)}|u(t)|^{6}\td x\right)^{\frac12}\|\nabla v_n(t)\|^2, ~~\forall t_m\leq t\leq T_n.
    \end{aligned}
  \end{equation}
  By the compactness of $\mathscr{A}$ again, there exists  $n_0 \in \mathbb{N}$ which is greater than $n_1+6C_5$ such that
  \begin{equation}\label{5.50}
    \sup_{t\in \R}\int_{A_{n_0}(t)}|u(t)|^6\td x\leq\left(\frac{\al}{8C_1C_7}\right)^{2}.
  \end{equation}
  Substituting  inequalities \eqref{5.49} and \eqref{5.50} into \eqref{5.44} gives
  \begin{equation}\label{new:50}
    \begin{aligned}
       & {\frac{\td}{\td t}}\left(\frac12\|\partial_{t} v_{n_0}(t)\|^2+\frac12\|\nabla v_{n_0}(t)\|^2+\al \langle \partial_{t} v_{n_0}(t),v_{n_0}(t)\rangle\right)+\frac\al8 \|\nabla v_{n_0}(t)\|^{2} \\
       & +\frac\al8\| \partial_{t} v_{n_0}(t)\|^{2} \leq C_2\al^2 \int_\Omega g_2(u_t,\partial_{t} w_{n_0})\td x \|\nabla v_{n_0}(t)\|^2, ~~ t_m\le t\le T_{n_0}.
    \end{aligned}
  \end{equation}

  Applying Gronwall's Lemma to \eqref{new:50} results in the exponential decay estimate of energy
  \begin{equation}\label{new:51}
    \|\partial_{t} v_{n_0}(t)\|^2+\|\nabla v_{n_0}(t)\|^2 \le C\left(\|\partial_{t} v_{n_0}(t_m)\|^2+\|\nabla v_{n_0}(t_m)\|^2\right)e^{-\frac{\al (t-t_m)}{8}},
  \end{equation}
  that holds on $[t_m, T_{n_0}]$. Finally, \eqref{5.41} follows directly from letting $m\to \infty$ in \eqref{new:51}. The proof is complete.
\end{proof}

By Lemmas \ref{lem5.5} and \ref{lem5.6},  $(u(T_{n_0}), u_t(T_{n_0})) \in \V$ and
\begin{equation*}
  \|(u(T_{n_0}), u_t(T_{n_0}))\|_{\V}\le r_0.
\end{equation*}
Since $u(t)$ satisfies problem \eqref{eq1} on $(T_{n_0}, \infty) \times \Omega$ with initial value $(u(T_{n_0}), u_t(T_{n_0}))$, applying Lemma \ref{unif_strogn_est}  gives $\|(u_0, u_1)\|_{\V} \le Q(r_0+\|\bar{u}_0\|_{H^2(\Omega)})$. Recall that all equilibrium points of $S(t)$ are uniformly bounded in $\V$,  $\mathscr{A}$ is a bounded subset of $\V$. Additionally, the bound of $\mathscr{A}$ can be refined using \eqref{Phi}, and the following theorem hold.

\begin{thm}[Regularity of the global attractor]\label{regularity}
  Assume that $\phi\in L^2(\Omega)$ and updated Assumption ({\bf G}) with conditions \eqref{G1}, \eqref{G2} and Assumption ({\bf F}) with conditions \eqref{F1}-\eqref{F3} hold. Then the global attractor $\mr{A}$ of problem \eqref{eq1} is bounded in $\V$.
\end{thm}
\begin{remark}
  The strategy in Section 4 and Section 5 also holds in the higher dimensions case. In the case of dimension $N\ge 4$, under the same assumptions declared in Remark \ref{re2} and condition \eqref{G2}, the solution semigroup $S(t)$ also possesses a global attractor which is bounded in $\V$.
\end{remark}
\section{Summary}\label{sec:sum}
In this paper, we investigated the well-posedness, global attractors, and regularity properties of weak solutions for a wave equation with nonlinear damping and super-cubic nonlinearity. By employing a priori estimates and developing a series of lemmas and theorems, we demonstrated the existence and uniqueness of weak solutions. We then established the existence of global attractors for the semigroup associated with the equation, and showed that these attractors are bounded in the appropriate function space. Further analysis provided insight into the regularity of the global attractors under specific conditions. This work contributes to the understanding of the long-term behavior and stability of solutions to complex wave equations with nonlinear characteristics.











\begin{thebibliography}{00}
  \bibitem{babin92}A.V. Babin, M.I. Vishik, Attractors of evolution equations, North-Holland, Amsterdam, 1992.
  \bibitem{ball04} J.M. Ball, Global attractors for damped semilinear wave equations, Discrete Contin. Dyn. Syst., 10 (2004) 31-52.

  \bibitem{chueshov02} V. Chepyzhov, M.I. Vishik, Attractors for equations of mathematical physics, volume 49 of American Mathematical Society Colloquium Publications, American Mathematical Society, Providence, RI, 2002.
  \bibitem{chueshov04} I. Chueshov, I. Lasiecka, Attractors for Second-Order Evolution Equations with a Nonlinear Damping, J. Dyn. Differ. Equations. 16 (2004) 469-512.
  \bibitem{chueshov08} I. Chueshov, I. Lasiecka, Long-time behavior of second order evolution equations with nonlinear damping, Mem. Am. Math. Soc. 195 (2008) 188.
  \bibitem{cl84}S. Ceron, O. Lopes. Existence of forced periodic solutions of dissipative semihnear hyperbolic equations and systems, Preprint, UNICAMP, Sao Paulo, 1984.
  \bibitem{94fei}E. Feireisl, Finite dimensional asymptotic behavior of some semilinear damped hyperbolic problems, J. Dynam. Differential Equations 6 (1) (1994) 23-35.
  \bibitem{feireisl95}E. Feireisl, Asymptotic behavior and attractors for a semilinear damped wave equation with supercritical exponent, Proc. Roy. Soc. Edinburgh Sect., 125A (1995) 1051-1062.
  \bibitem{feireisl95_nldamp} E. Feireisl, Global Attractors for Semilinear Damped Wave Equations with Supercritical Exponent, J. Differ. Equ. 116 (1995) 431-447.
  \bibitem{grasselli04} M. Grasselli, V. Pata, Asymptotic behavior of a parabolic-hyperbolic system, Commun. Pure Appl. Anal. 3 (2004) 849-881.
  \bibitem{hale88} J.K. Hale, Asymptotic Behavior of Dissipative Systems, American Mathematical Society, Providence, RI, 1988.
  \bibitem{haraux87}A. Haraux, Semi-linear Hyperbolic Problems in Bounded Domains, Mathematical Reports, vol.3, Harwood Gordon Breach, New York, 1987.
  \bibitem{la02}I. Lasiecka, A.R. Ruzmaikina, Finite dimensionality and regularity of attractors for 2-D semilinear wave equation with nonlinear dissipation, J. Math. Anal. Appl. 270 (2002) 16-50.
  \bibitem{Lions65} J.L. Lions, W.A. Strauss, Some non-linear evolution equations, Bulletin de la Soci\'et\'e Math\'ematique de France. 93 (1965) 43-96.
  \bibitem{khanmamedov06} A.K. Khanmamedov, Global attractors for wave equations with nonlinear interior damping and critical exponents, J. Differ. Equ. 230 (2006) 702-719.
  \bibitem{Khanmamedov06_karman}A.K. Khanmamedov, Global attractors for von Karman equations with nonlinear interior dissipation, Journal of Mathematical Analysis and Applications. 318 (2006) 92-101.

  \bibitem{khanmamedov09} A.K. Khanmamedov, On a global attractor for the strong solutions of the 2-D wave equation with nonlinear damping, Appl. Anal. 88 (2009) 1283-1301.
  \bibitem{khanmamedov10} A.K. Khanmamedov, Remark on the regularity of the global attractor for the wave equation with nonlinear damping, Nonlinear Anal. Theory, Methods Appl. 72 (2010) 1993-1999.
  \bibitem{Khanmamedov10-dis} A.Kh. Khanmamedov, A strong global attractor for the 3D wave equation with displacement dependent damping, Applied Mathematics Letters. 23 (2010) 928-934.
  \bibitem{mrw98} I. Moise, R.M.S. Rosa, X. Wang, Attractors for non-compact semigroups via energy equations, Nonlinearity. 11 (1998) 1369-1393.
  \bibitem{nakao06} M. Nakao, Global attractors for nonlinear wave equations with nonlinear dissipative terms, J. Differ. Equ. 227 (2006) 204-229.

  \bibitem{radu13} P. Radu, Strong solutions for semilinear wave equations with damping and source terms, Appl. Anal. 92 (2013) 718-739.
  \bibitem{rau92}G. Raugel, Une equation des ondes avec amortissment non lineaire dans le cas critique en dimensions trois, C. R. Acad. Sci. Paris Ser. I 314 (1992) 177-182.

  \bibitem{simon87}J. Simon, Compact Sets in the Space $L^p(0,T;B)$, Ann. Di Mat. Pura Ed Appl., CXLVI (1987) 65-96.

  \bibitem{sun06} C. Sun, M. Yang, C. Zhong, Global attractors for the wave equation with nonlinear damping, J. Differ. Equ. 227 (2006) 427-443.
  \bibitem{temam97} R. Temam, Infinite-dimensional dynamical systems in mechanics and physics, Second edition. Applied Mathematical Sciences, 68. Springer-Verlag, New York, 1997.
  \bibitem{todorova15} G. Todorova, B. Yordanov, On the regularizing effect of nonlinear damping in hyperbolic equations, Trans. Am. Math. Soc. 367 (2015) 5043-5058.
\end{thebibliography}
\end{document}